\documentclass{article}

\usepackage{arxiv}

\usepackage[utf8]{inputenc} 
\usepackage[T1]{fontenc}    
\usepackage{hyperref}       
\usepackage{url}            
\usepackage{booktabs}       
\usepackage{amsfonts}       
\usepackage{nicefrac}       
\usepackage{microtype}      
\usepackage{lipsum}
\usepackage{graphicx}
\usepackage{xcolor}

\graphicspath{ {./images/} }

\usepackage{amsmath}
\usepackage{amssymb}
\usepackage{enumerate}

\usepackage{amsthm}

\newtheorem{theorem}{Theorem}[section]
\newtheorem{lemma}[theorem]{Lemma}
\newtheorem{hypothesis}[theorem]{Hypothesis}

\theoremstyle{definition}
\newtheorem{definition}[theorem]{Definition}

\theoremstyle{remark}
\newtheorem{remark}[theorem]{Remark}

\numberwithin{equation}{section}

\DeclareMathOperator{\vspan}{span}
\newcommand{\eps}{\varepsilon}
\newcommand{\ssubset}{\subset\joinrel\subset}
\newcommand{\norm}[1]{\left\lVert#1\right\rVert}
\newcommand{\abs}[1]{\left\lvert#1\right\rvert}  

\DeclareMathOperator{\coker}{coker}

\newcommand{\N}{\mathbb{N}}

\newcommand{\R}{\mathbb{R}}

\newcommand{\supp}{\mathrm{supp}}

\title{Existence of Weak Solutions for a Nonlocal Klausmeier Model}

\author{
Gabriela Jaramillo \\
Department of Mathematics\\
University of Houston\\
Houston, TX 77004\\
\texttt{gabriela@math.uh.edu} 
\thanks{This work is supported by NSF grant DMS-2307500 (G.J., C.M.)}
\And
Cristian Meraz \\
Department of Mathematics\\
University of Houston\\
Houston, TX 77004\\
\texttt{cmeraz@uh.edu}
\thanks{Corresponding author.}
\\
\\
}

\begin{document}
\maketitle
\begin{abstract}
We establish the existence of weak solutions for a nonlocal Klausmeier model within a small time interval $[0, T)$.
The Klausmeier model is a coupled,  nonlinear system of partial differential equations governing plant biomass and water dynamics in semiarid regions. The original model posits that plants disperse their seed according to classical diffusion. Instead, we opt for a nonlocal diffusive operator in alignment with ecological field data that validates long-range dispersive behaviors of plants and seeds. 
The equations, defined on a finite interval in $\mathbb{R}$, feature homogeneous Dirichlet boundary conditions for the water equation and nonlocal Dirichlet volume constraints for the plant biomass equation. The nonlocal operator involves convolution with a symmetric and spatially extended convolution kernel possessing mild integrability and regularity properties. We employ the Galerkin method to establish the existence of weak solutions. The key challenge comes from the nonlocal operator; we define it on a subspace of $L^{2}$ instead of $H^{1}$, precluding the use of Aubin's compactness theorem to prove the weak convergence of nonlinear terms. To overcome this, we modify the model and introduce two new equations for the spatial derivatives of plant biomass and water. This procedure allows us to recover enough regularity to establish compactness and complete the proof.
\end{abstract}


\section{Introduction}

The Klausmeier model and its variations \cite{klausmeier1999, vanderStelt2013, sewalt2017, sherratt2018, bastiaansen2019}  provide a relatively simple mathematical description of arid ecosystems
that is both capable of reproducing the vegetation patterns seen in nature \cite{BROMLEY19971, AGUIAR1999273, LEPRUN199925, couteron2001, macfadyen1950, klausmeier1999, meron2001, wilcox2003}, as well as being amenable to mathematical analysis \cite{Sherratt_2010, sherratt2011, sherratt2013, sherratt2013_2, sherratt2013_3,vanderStelt2013, sewalt2017, sherratt2018, bastiaansen2019}. 
Interest in these models stems from a desire to understand the relation between the formation of vegetation patterns and possibly irreversible changes to these ecosystems, including complete loss of vegetation. 
In this paper we consider a nonlocal version of these equations where the spread of plants is described using a convolution operator \cite{sherratt2018}. With an eye towards developing a  finite element scheme to simulate this system of now integro-differential equations, we pose the equations on a bounded 1-d domain with nonlocal Dirichlet boundary constraints and prove the existence of small-time weak solutions using a modified Galerkin method.

 To motivate our model equations, we start our discussion with the original system used by  Klausmeier \cite{klausmeier1999}, which after a non-dimensional analysis takes the form,
 \begin{equation}
\label{eqn:original_Klaus}
\begin{split}
u_t & = d \Delta u + u^2 w - \mu w\\
w_t  & = \nu w_{x} + a - w - u^2 w.
\end{split}
\end{equation}
Here, $u$ represents the density of plant biomass, while $w$ represents the water content in the soil. The original equations are then posed on the plane, so that $u$ and $v$ are functions of $(x,t) \in \R^2 \times \{t \geq 0\}$.
In the water equation,
 the parameter $\nu$ encodes the slope in the terrain, so that water flows downhill in the negative $x$-direction if $\nu >0$. 
 The constant $a$ then describes the amount of rainfall that reaches the system, 
 the term $-w$ represents water loss due to evaporation,
  and the expression $- u^2 w$ models water uptake by plant roots. Notice that this last term describes the positive feedback created by plants, whose presence promotes water infiltration into the soil and thus facilitates water absorption. 
The model then assumes that plant growth is proportional to water uptake and is therefore given by $ u^2w$, while
the parameters $d$ and $\mu$ represent the rate of plant dispersal and plant mortality, respectively.

The above equations support solutions representing striped patterns and were originally used to model the tiger bush patterns seen on sloped terrains in semi-desert areas \cite{macfadyen1950,worral1959,montana1990}. From a mathematical point of view these patterns arise as a result of the interplay between the advection term in the water equation and the nonlinearity, $u^2w$, which together model a scale-dependent feedback. They describe how the local flow of water towards vegetation comes at the cost of available water at larger scales. As a result, the system \eqref{eqn:original_Klaus} does not support the existence of striped patterns when $\nu =0$. Since these results are inconsistent with field studies showing spots and labyrinth structures on flat terrains \cite{belsky1986population, white1970brousse, dunkerley2002oblique},
alternative versions of the Klausmeier model include a Laplacian term describing the flow of water into the soil via diffusion. See also \cite{vanderStelt2013, sewalt2017, gilad2007mathematical, gilad2004ecosystem, bennett2019long, vonhardenberg2001diversity, hillel2014environmental}.

A second modification to the Klausmeier model used in the literature is based on \cite{bastiaansen2019}, where the authors assume that plant growth is limited not only by water availability, but also by soil resources. As a result, plant growth depends on the local density of plant biomass and is thus modeled using a logistic-type growth term of the form $ (1-bu) u^2w$.

Finally, while in the original system plant dispersal is described as a local diffusion process, and thus modeled using the Laplacian, numerous field studies show that this process is actually nonlocal (see \cite{bullock2017synthesis} and references therein). Indeed, for many species, the spatial spread of plants is a consequence of seed dispersal by wind or animals, and is therefore better described using an integral operator of the form,
\[ K u = \int_ \R (u(y) - u(x)) \gamma(x,y) \; dy.\]
Here, the kernel $\gamma(x,y)$ can be viewed as a probability density function describing the likelihood that seeds originating at position $y$ land at location $x$, see \cite{pueyo2008}. Consequently,  the operator $K$ describes the flow in and out of position $x$ and can thus be viewed as a nonlocal form of diffusion \cite{du2012analysis}. Here we consider spatially extended kernels that are positive, symmetric and live in $W^{1,1}(\R) \cap H^1(\R)$. In addition, we assume that as a function of $|z|=|x-y|$ these kernels have finite second moments.

Our nonlocal Klausmeier model therefore takes the form,
\begin{equation}
\label{eqn:model-explain}
\begin{aligned}
u_{t} &= \overbrace{d K u }^{\stackrel{ \text{nonlocal} }{ \text{plant dispersal} }} + \overbrace{u^{2} w (1 - b u)}^{\stackrel{\text{plant growth due to}}{\text{water uptake}}} - \overbrace{\mu u}^{\text{plant loss}} ,
\\
w_{t} &= \underbrace{w_{xx}}_{\stackrel{\text{water diffusion}}{\text{through soil}}} + \underbrace{\nu w_{x}}_{\stackrel{\text{water flow}}{\text{downhill}}} + 
\underbrace{a}_{\text{rainfall}} - \underbrace{w}_{\text{evaporation}} - \underbrace{u^{2} w}_{\stackrel{\text{water uptake}}{\text{by plants}}} , \\ 
\end{aligned}
\end{equation}
where the operator $K$ is defined as above. For simplicity and ease of exposition the equations are posed on a bounded interval in $\R$, $\Omega =[-L,L]$ with appropriate boundary conditions. For the water equation we use homogeneous Dirichlet boundary conditions, modeling a desert state on the boundary of $\Omega$. 
For the plant equation we use instead nonlocal Dirichlet boundary constraints and require $u =0$ on $\Omega^c = \R \setminus \Omega$. This type of boundary constraint is necessary given the structure of our operator $K$. Since this map is defined by an integral defined over the whole real line, we must have information about $u$ on all of $\R$ and not only on the domain of interest $\Omega$. The choice of homogeneous nonlocal Dirichlet constraints is then consistent with our assumption of having a dessert state surrounding our domain $\Omega$.

As mentioned above, our goal in this paper is to prove the existence of small-time weak solutions for the system of equations \eqref{eqn:model-explain}. One of the main difficulties we face comes from the convolution operator, \( K \). Unlike the Laplacian, \( K \) does not define a bilinear form with domain \( H^{1}(\Omega) \times H^{1}(\Omega) \), but rather with domain \( L^{2}(\Omega) \times L^{2}(\Omega) \). 
Hence, the sequence of approximations generated by the Galerkin method belong to a space that does not have enough regularity. Consequently, we cannot use Aubin's Compactness Theorem to show that the nonlinear terms in the equations converge strongly in \( L^{2}(\Omega) \), which would then allow us to prove weak convergence of the nonlinear terms. (For a statement of Aubin's compactness theorem, see Section \ref{sec:3}.)
The approach we take here is to extend the system \eqref{eqn:model-explain} and include auxiliary equations for two additional variables, \( v \) and \( z \). We then show that, as approximate solutions, the auxiliary variables correspond to the derivatives of the Galerkin approximations for \( u \) and \( w \), and thus recover the required regularity to carry out the rest of the proof. 
We summarize the result of this procedure in the following theorem.

\begin{theorem}[Main Theorem]
\label{thm:main}
Let $\gamma(z)$ satisfy Hypothesis \ref{hyp:kernel} and take \( u_{0}, w_{0} \) to be given functions in $H^1(\Omega) \times H^1_0(\Omega)$. Then, 
there exist positive constants \( \mathcal{C}_{1} \), \( \mathcal{C}_{2} \), and \( T \), such that if 
\begin{equation*}
 \norm{ u_{0} }_{L^{2}(\Omega)} + \norm{ w_{0} }_{L^{2}(\Omega)} < \mathcal{C}_{1}, \quad 
 \norm{ \partial_{x} u_{0} }_{L^{2}(\Omega)} + \norm{ \partial_{x} w_{0} }_{L^{2}(\Omega)} < \mathcal{C}_{2}, 
\end{equation*}
 the weak formulation of equations \eqref{eqn:model-explain}, i.e.
\begin{align*}
 \left\langle u_{t}, \varphi \right\rangle_{\left( V^{\prime} \times V \right)} + d B[u, \varphi] + \mu \left( u, \varphi \right)_{L^{2}(\Omega)} &= 
 \left( u^{2} w (1 - b u), \varphi \right)_{L^{2}(\Omega)}, \\ 
 \left\langle w_{t}, \psi \right\rangle_{\left( H^{-1}(\Omega) \times H^{1}_{0}(\Omega) \right)} + G[w, \psi]+ \left( w, \psi \right) &= 
 \left( a - u^{2} w, \psi \right)_{L^{2}(\Omega)} ,
\end{align*}
has a unique solution \( (u,w) \in L^{2}\left( 0,T; H^{1}(\Omega) \right) \times L^{2}\left( 0,T; H^{1}(\Omega) \right) \) satisfying the initial conditions, \( u(x,0) = u_{0} \) and \( w(x,0) = w_{0} \).
\end{theorem}

Here, the space $V \sim L^2(\Omega)$ denotes the domain of the operator $K$, while $V'$  is its dual. A more complete definition of these spaces, together with a statement of Hypothesis \ref{hyp:kernel} and properties of the bilinear forms $B$ and $G$,  is provided in Subsection \ref{sub:notation}.

\begin{remark}
We close this introduction with some remarks regarding extensions of the above result.

\begin{enumerate}[i)]
\item As in \cite{cappanera2024analysis}, the modified Galerkin method presented here can also be used to prove the short-time existence of weak solutions for system \eqref{eqn:model-explain} subject to zero flux boundary constraints. In the case of the plant equation, where the spread of plant biomass is given by a nonlocal form of diffusion, this type of boundary conditions have to be modeled using nonlocal Neumann boundary constraints. 
Assuming that the convolution kernel, $\gamma$,  has compact support, this is equivalent to requiring
\[ K u = 0 \quad \mbox{for} \quad x \in \Omega_o = \tilde{\Omega} \setminus \Omega\]
where
\[ \tilde{\Omega} = \{ y \in \R \mid z= (x-y)  \in \supp \;\{ \gamma(z)\}, \quad  \forall x \in \Omega\}\]
In particular, one can show that the space 
\[ V_N = \{ u \in L^2(\tilde{\Omega}) \mid  K u (x) = 0 \quad x \in \Omega_o\}\]
which  represents the domain of the operator $K$, is isomorphic to $L^2(\Omega)$ (see \cite{burkovska2021, cappanera2024analysis}). Moreover, as shown in \cite{cappanera2024analysis} it is possible to construct a set of orthonormal basis functions for $V_N$, so that the method described in the following sections applies to this setting as well.

\item Because of the logistic term included in our nonlocal Klausmeier model, \eqref{eqn:model-explain}, we expect that solutions that start from nonnegative initial conditions remain positive and stay bounded. While we do not show that these properties hold,   it is  possible to use instead methods from semigroup theory to prove global existence of positive mild solutions belonging to $ L^\infty(\Omega) \times L^\infty(\Omega)$. See \cite{Philippe2023} for an example of this approach in the context of a nonlocal Gray-Scott model. We plan to address this in future work.

\item 
Given that the space $V_N$ is isomorphic to $L^2(\Omega)$, it is possible to show that the operator $K: V_N \longrightarrow L^2(\Omega)$ defines a semigroup. However, the lack of regularity of the domain again becomes an issue, as it precludes us from proving existence of mild solutions using the standard approach of combining properties of the semigroups together fixed point arguments. On the other hand, it is not clear if there is an appropriate dense subspace $D \subset L^\infty(\Omega)$ satisfying nonlocal Neumann boundary constraints such that the operator $K: D \subset L^\infty(\Omega)  \longrightarrow L^\infty(\Omega)$ defines a semigroup. As a result, to our knowledge, the modified Galerkin approach presented here seems to be the only method for showing existence solutions for nonlocal problems of the form \eqref{eqn:model-explain} in the case of zero flux boundary constraints.

\end{enumerate}
\end{remark}

We finish this section by establishing the notation used throughout the paper and stating our main assumptions.

\subsection{Notation and Model equations}
\label{sub:notation}

We consider the nonlocal Klausmeier model 
\begin{equation}
\label{eqn:model}
\begin{aligned}
u_{t} &= d K u + u^{2} w (1 - b u) - \mu u  ,  \\
w_{t} &= w_{xx} + \nu w_{x} + a - u^{2} w  - w   ,
\end{aligned}
\end{equation}
and assume that the parameters  \( a, b, d, \nu \) and \( \mu \) are positive. The map $K$ is then defined via the integral,
\[ K u = \int_\R (u(y) - u(x)) \gamma(x,y) \;dy\]
where the convolution kernel satisfies the following assumptions.

\begin{hypothesis}
\label{hyp:kernel}
The convolution kernel, \( \gamma(z)=\gamma(\abs{z}) \), is a radial function in \( H^{1}(\R) \cap W^{1,1}(\R) \), which is strictly greater than zero on all of \( \R \) and satisfies:
\begin{itemize}
\item \( \gamma(x,y) = \gamma(\abs{x-y}) = \gamma (\abs{z})\), and 

\item \( \int_{\R} z^{2} \gamma(z) \, dz < \infty \).

\end{itemize}
\end{hypothesis}

We pose equations \eqref{eqn:model} on a bounded domain $\Omega =[-L,L]$ and consider homogeneous Dirichlet type boundary conditions.
In particular, we choose 
\begin{align}
 \label{eqn:vc-plant}
 u &= 0 \quad \text{in } \Omega^{c}, \\ 
 \label{eqn:bc-water}
 w &= 0 \quad \text{on } \partial\Omega .
\end{align}

Then, due to the volume constraints imposed on the plant equation, we view the map $K: V \subset L^2(\R) \longrightarrow L^2(\R)$ as an operator with domain 
\begin{equation}
\label{eqn:V} 
V = \{ u \in L^{2} (\R) \mid u (x) = 0 \quad \text{for a.e. } x\in \Omega^{c} \} .
\end{equation}
and equip this space with the inner product from \( L^{2}(\R) \). Notice then that the space $V$ is isometrically isomorphic to $L^2(\Omega)$. That is, there exists a linear isometric isomorphism \( T: V \to L^{2}(\Omega) \).

{\bf Outline:} The rest of the paper is organized as follows: In Section 2, we describe the bilinear forms appearing in the weak formulation of the model (\ref{eqn:model-explain}) and some of their properties; in Section 3, we develop the modified Galerkin approach necessary to prove Theorem \ref{thm:main}. In Section 4, energy estimates that allow us to proceed with weak convergence of approximations to solutions are proved. Section 5 then contains details of this convergence, and we conclude with a proof of the main theorem, Theorem \ref{thm:main}. We close the paper by discussing our results and procedure and providing future directions.

\section{Preliminary Results}
\label{sec:2}

We briefly outline the contents of Chapter \ref{sec:2}. 
In Section 2.1, we establish \( K \) as a bounded linear operator on the space \( V \) (see equation (\ref{eqn:V})). Next, in Section 2.2, we define the bilinear forms associated with the linear operators \( K \) and \( \partial_{xx} + \nu \partial_{x} \) appearing in the model (\ref{eqn:model}) and give the results concerning boundedness and coercivity for the bilinear forms. 
Finally, we describe a relation satisfied by the bilinear form \( B[\cdot, \cdot] \) defined in (\ref{eqn:bilinear-form-plant}) that will be useful in proving energy estimates.

\subsection{Bilinear forms}
\label{sec:2_1}

\begin{definition}
We define the bilinear form \( B: V \times V \to \R \) using the expression 
\begin{equation}
\label{eqn:bilinear-form-plant}
B[u, \varphi] 
= 
\frac{1}{2} \int_{\R} \int_{\R} (u(y) -u(x))  \gamma(x,y) (\varphi(y) - \varphi(x)) \, dy \, dx  ,
\end{equation}
where \( \gamma \) satisfies Hypothesis \ref{hyp:kernel}.
\end{definition}

\begin{definition}
We define the bilinear form \( G : H \times H \to \R \) by
\begin{equation}
\label{eqn:bilinear-form-water}
G[w,\psi] 
=  
\int_{\Omega} w_{x} \psi_{x} - \nu w_{x} \psi  \, dx ,
\end{equation}
where $H$ is either $H^{1}_{0}(\Omega)$, or 
\begin{equation}
\label{eqn:W}
    H = W := \left\{ w \in H^{1}(\Omega) \mid \int_{\Omega} w \, dx = 0 \right\} .
\end{equation}
\end{definition}

\begin{remark}
 Although we pose the water equation on $\Omega$ with Dirichlet boundary conditions, because of the modified Galerkin system described in subsection \ref{sec:3_1}, we employ the bilinear form \( G[\cdot,\cdot] \) in the case of an equation satisfying Neumann boundary conditions. Note that the space $W$ is a Hilbert space with norm and inner product from $H^{1}(\Omega) $.
\end{remark}

\subsection{Properties of the bilinear forms}

The next three lemmas state that the bilinear forms \( G: H\times H \to \R \) and \( B : V\times V \to \R \), defined in (\ref{eqn:bilinear-form-water}) and (\ref{eqn:bilinear-form-plant}), respectively, are bounded and coercive. 

\begin{lemma}
\label{lem:bilinear-form-water}
There exist positive constants \( C \), \( \kappa \), and \( \chi \), depending on the positive parameter \( \nu \) appearing in the system (\ref{eqn:model}) and the domain \( \Omega \), satisfying
\begin{itemize}
\item[i)]
\( G [w,\psi] \leq C \norm{w}_{H^{1}_{0}(\Omega)} \norm{\psi}_{H^{1}_{0}(\Omega)} \),

\item[ii)]
\( \kappa \norm{w}_{H^{1}_{0}(\Omega)} ^{2} \leq G[w,w] + \chi \norm{w}_{L^{2}(\Omega)}^{2} \),
\end{itemize}
for all \( w, \psi \in H^{1}_{0}(\Omega) \). 
\end{lemma}

\begin{lemma}
\label{lem:bilinear-form-water-W}
There exist positive constants \( C \), \( \kappa^{\prime} \), and \( \chi^{\prime} \), depending on the positive parameter \( \nu \) appearing in the system (\ref{eqn:model}) and the domain \( \Omega \), satisfying
\begin{itemize}
\item[i)]
\( G [z, \rho] \leq C \norm{z}_{W} \norm{\rho}_{W} \),

\item[ii)]
\( \kappa^{\prime} \norm{z}_{W} ^{2} \leq G[z,z] + \chi^{\prime} \norm{z}_{L^{2}(\Omega)}^{2} \),
\end{itemize}
for all \( z, \rho \in W \). 
\end{lemma}

The proofs of Lemmas \ref{lem:bilinear-form-water}
 and \ref{lem:bilinear-form-water-W} are standard and found in the literature (see, e.g., \cite{evans2022partial}). We omit them here.
 
\begin{lemma}
\label{lem:bilinear-form-plant}
There exist positive constants \( C \) and \( \theta \) satisfying
\begin{itemize}
\item[i)] \( B [u,\varphi] \leq C \norm{u}_{L^{2}(\Omega)} \norm{\varphi}_{L^{2}(\Omega)} \), 

\item[ii)] \( B [u,u] \geq \theta \norm{u}_{L^{2}(\Omega)}^{2} \),
\end{itemize}
for all \( u, \varphi \in V \).
\end{lemma}

\begin{proof}
Since $\gamma (z) \in L^1(\R)$, using Young's inequality for integrals we have that the operator $K: V \longrightarrow L^2(\R)$  is bounded. 

Below, we state Lemma \ref{lem:formula-inner-product-bilinear}, which gives us the identity \( B[u,\varphi] = - (K u, \varphi)_{L^{2}(\Omega)} \). Then, by Cauchy-Schwarz:
\begin{align*}
B[u,\varphi] &\leq \norm{Ku}_{L^{2}(\Omega)} \norm{\varphi}_{L^{2}(\Omega)} \\ 
&\leq C \norm{u}_{L^{2}(\Omega)} \norm{\varphi}_{L^{2}(\Omega)} ,
\end{align*}
where \( C = \norm{K}_{\text{op}} = \sup_{\norm{u}_{L^{2}(\Omega)}=1} \norm{Ku}_{L^{2}(\Omega)} < \infty \) is the operator norm of \( K \).

To show that \( B \) is coercive, we first show that \( - K \) is invertible and then argue via a proof by contradiction.
Then, because of Hypothesis \ref{hyp:kernel}, we have 
\begin{align*}
\int_{\Omega} \int_{\Omega} \abs{ \gamma(x,y) }^{2} \, dy \, dx &\leq \int_{\Omega} \int_{\R} \abs{ \gamma(x,y) }^{2} \, dy \, dx \\ 
&= \abs{\Omega} \norm{\gamma}_{L^{2}(\R)}^{2} < \infty .
\end{align*}
This shows \( \gamma \in L^{2}(\Omega\times \Omega) \), and it follows that the convolution with \( \gamma \) is a Hilbert-Schmidt operator on \( L^{2}(\Omega) \), hence compact. 

Therefore, writing \( - K u = \Gamma u - \gamma * u \), we see that \( - K \) is a compact perturbation of a constant multiple of the identity. Since the identity is a Fredholm operator with index zero and Fredholm properties are preserved under compact perturbations, we conclude that \( - K \) is also Fredholm with index zero. We now show that the nullspace of \( K \) is trivial; this, together with the above, implies invertibility.

Assume that there exists a function \( u \in V \) satisfying \( K u= 0 \). This implies \( -\left( Ku, u \right)_{L^{2}(\Omega)} = 0 \). On the other hand, using Lemma \ref{lem:formula-inner-product-bilinear},
\begin{align*}
 - \left( K u , u \right)_{L^{2}(\Omega)} &= \int_{\R} \int_{\R} \left( u(x)-u(y)  \right)^{2} \gamma(x,y) \, dy \, dx = 0 .
\end{align*}
By Hypothesis \ref{hyp:kernel}, we have that \( \gamma(x,y) > 0 \) for all \( (x,y) \in \R^{2} \).
In particular, the above computation shows that one must have \( \left( u(x)-u(y) \right)^{2} = 0 \) for a.e. \( (x,y) \in \R^{2} \), and it then follows that \( u \) is constant a.e. in \( \R \). Then since \( u = 0 \) on \( \Omega^{c} \), we have that \( u =0 \) a.e. in \( \R \) is a necessary condition for \( Ku = 0 \). Hence, \( \dim \ker K = 0 \) and also \( \dim \coker K = 0 \), and \( K \) is invertible.

Then by the Open Mapping Theorem, the inverse of \( K \) is bounded in \( L^{2}(\Omega) \). Hence, there exists \( C >0 \) such that \( \|K^{-1} f \| \leq C \|f\| \), which is equivalent to \( \norm{u} \leq C\| Ku\| \) provided \( Ku = f \).

Now assume that \( B \) is not coercive. Namely, suppose there exists a sequence \( (u_{n}) \subset L^{2}(\Omega) \) satisfying \( \norm{u_{n}}=1 \) and \( B[u_{n},u_{n}] \to 0 \). 

To derive a contraction, first note that by Cauchy-Schwarz,
\begin{align*}
B[u,\varphi] &= \frac{1}{2} \int_{\R} \int_{\R} (u(y) -u(x)) (\varphi(y) - \varphi(x)) \gamma(x,y) \, dy \, dx \\ 
&= \frac{1}{2} \int_{\R} \int_{\R}  \left( (u(y) -u(x)) \left[ \gamma(x,y) \right]^{1/2} \right) \left( (\varphi(y) - \varphi(x)) \left[ \gamma(x,y) \right]^{1/2} \right) \, dy \, dx  \\
&\leq \frac{1}{2}\left(  \int_{\R} \int_{\R}  \left( (u(y) -u(x))^{2} \gamma(x,y)  \right) \, dy \, dx  \right)^{1/2}
\\ 
&\qquad \cdot \left( \int_{\R} \int_{\R} \left( (\varphi(y) - \varphi(x))^{2}  \gamma(x,y) \right)   \, dy \, dx  \right)^{1/2} \\
&= \frac{1}{2} \left[ B[u,u] \right]^{1/2} \left[ B[\varphi, \varphi] \right]^{1/2} .
\end{align*}

Then 
\begin{align*}
(K u_{n}, K u_{n}) 
&= \norm{K u_{n}}^{2} \\
&= B[ -u_{n}, K u_{n}] \\
&\leq \frac{1}{2} \left[ B[u_{n},u_{n}] \right]^{1/2} \left[ B[K u_{n}, K u_{n}] \right]^{1/2} .
\end{align*}
This implies that
\begin{align*}
\norm{K u_{n}}^{2} &\leq 
\frac{1}{2} \left[ B[u_{n},u_{n}] \right]^{1/2} 
\|K u_{n} \| \\ 
\norm{K u_{n}} &\leq 
\frac{1}{2} \left[ B[u_{n},u_{n}] \right]^{1/2} 
\end{align*}
Therefore, \( K u_{n} \to 0 \) in \( L^{2}(\Omega) \) because \( B[u_{n}, u_{n}] \to 0 \). 

Then
\[
1 = \norm{u_{n}} \leq \norm{K u_{n}} \to 0,
\] 
a contradiction. We conclude the existence of a constant \( \theta > 0 \) satisfying 
\[
B [u,u] \geq \theta \norm{u}_{L^{2}(\Omega)}^{2},
\]
and consequently, that \( B \) is a coercive bilinear form.

\end{proof}

In proving energy estimates, it is convenient to have the following relation:
\begin{equation}
\label{eqn:bilinear-form-plant-ip}
B[u,u] = - \left( K u, u \right)_{L^{2}(\Omega)} ,
\end{equation}
For this, we have the following lemma.
\begin{lemma}
\label{lem:formula-inner-product-bilinear}
The equality
\begin{equation*}
- \left( Ku, \varphi \right)_{L^{2}(\Omega)}  = \frac{1}{2} \int_{\R} \int_{\R} \left( u(y) - u(x) \right) \gamma(x,y) \left( \varphi(y) - \varphi(x) \right) \, dy \, dx 
\end{equation*}
holds for all \( u, \varphi \in V \), where \( V \) is defined in (\ref{eqn:V}).
\end{lemma}

In particular, the operator \( -K \) is self-adjoint on \( V \sim L^{2}(\Omega) \).
Moreover, thanks to Lemma \ref{lem:bilinear-form-plant}, we also have that  \( -K \) is positive definite.

\section{Galerkin Method}
\label{sec:3}

We seek to prove the existence of weak solutions of the nonlocal Klausmeier model (\ref{eqn:model}).
Thus, we seek a time \( T >0 \) and a tuple
\[
 (u,w) \in L^{2}\left( 0,T, V \right)  \times  L^{2}\left( 0,T, H^{1}_{0}(\Omega) \right) 
\]
with 
\[
 (u_{t},w_{t}) \in  L^{2}\left( 0,T,V^{\prime} \right) \times  L^{2}\left( 0,T,H^{-1}(\Omega) \right) 
\] 
satisfying
\begin{equation}
\label{eqn:model-weak}
\begin{aligned}
\left< u_{t}, \varphi \right>_{(V^{\prime}\times V)} + d B[u,\varphi] + \mu \left( u, \varphi \right)_{L^{2}(\Omega)} &= 
\left( u^{2}w (1-b u), \varphi \right)_{L^{2}(\Omega)} , \\
\left< w_{t}, \psi \right>_{(H^{-1}(\Omega) \times H^{1}_{0}(\Omega))} + G[w, \psi] + \left( w, \psi \right)_{L^{2}(\Omega)} &= 
\left(a - u^{2} w, \psi \right)_{L^{2}(\Omega)} ,
\end{aligned}
\end{equation}
for each \( (\varphi, \psi) \in V \times H^{1}_{0}(\Omega) \) and a.e. time \( 0 \leq t < T \), as well as the initial conditions 
\[
u(x,0) = u_{0}(x), \quad
w(x,0) = w_{0}(x) ,
\]
for a.e. \( x \in \Omega \).

Our approach relies on the Galerkin method and Aubin's Compactness Theorem, which we state next. See \cite{aubin1963} for a proof of this theorem.

\begin{theorem}[Aubin's Compactness Theorem]
\label{thm:aubin}
Let \( X_{0}, X, X_{1} \) be three Banach spaces such that 
\[
X_{0} \subset X \subset X_{1},
\]
where the injections are continuous, \( X_{0} \) and \( X_{1} \) are both reflexive, and the injection \( X_{0} \subset X \) is compact. Let \( T>0 \) be a fixed finite number, let \( 1 < p, q < \infty \), and define 
\[
\mathcal{Y} := \left\{ u \in L^{p} \left(0,T; X_{0}\right) : u' = \frac{du}{dt} \in L^{q} (0,T; X_{1}) \right\}.
\]
Then, \( \mathcal{Y} \) is a Banach space when it is equipped with the norm 
\[
\norm{u}_{\mathcal{Y}} = \norm{u}_{L^{p}(0,T;X_{0})} + \|u'\|_{L^{q}(0,T;X_{1})}.
\]
Furthermore, the injection \( \mathcal{Y} \) into \( L^{q}(0,T; X) \) is compact.
\end{theorem}

Notice that if $V = H^1(\Omega)$,  energy estimates would bound the sequence of Galerkin approximations \( u^{m} \) in \( X_0 = H^{1}(\Omega) \) for every $t \in [0,T]$. Aubin's Theorem then allows us to infer the existence of a  subsequence, \( u^{m_{l}} \), that  converges strongly in \(L^2(0,T;  X ) \) with \(X= L^{2}(\Omega) \). From this one can then show point-wise convergence of a subsequence followed by the Dominated Convergence theorem to finally arrive at the weak convergence of the nonlinear terms. 

In our equations, however, we have that $V = L^2(\Omega)$, so that we are only able to bound the sequence of Galerkin approximations $u^m$ in $L^2(\Omega)$. Consequently, our problem does not satisfy the assumptions of Theorem \ref{thm:aubin}.
To gain extra regularity for the approximate solutions $u^m$, we append two well-chosen additional equations to our nonlocal Klausmeier model for variables $v$ and $z$. We then show that the Galerkin approximations to this augmented system of four equations satisfy $v^m = \partial_x u^m$ and $z^m = \partial_x w^m$. Energy estimates then prove that the sequences $\{ u^{m} \}$ and $\{\partial_{x} u^{m} \}$ are bounded in $L^2(\Omega)$, allowing us to invoke Aubin's theorem.

\subsection{Augmented system}
\label{sec:3_1}

Consider the following system of four equations on the variable $u,w,v,z$,

\begin{equation}
\label{eqn:model-aug}
\begin{aligned}
 u_{t} &= d K u - \mu u + P^{(1)}(u,w) \\ 
 w_{t} &= w_{xx} + \nu w_{x}  - w  + P^{(2)}(u,w) \\
 v_{t} &= d J u - ( \mu+ d \Gamma ) v  +Q^{(1)}(u,w,v,z) \\ 
 z_{t} &= z_{xx} + \nu z_{x} - z + Q^{(2)}(u,w,v,z), \\ 
 \end{aligned}
\end{equation}
where 
\begin{equation}
\label{eqn:J}
 J u =
  \int_{\R} \left( u(y) - u(x) \right) \partial_{x} \gamma(x,y) \, dy, 
  \end{equation}
\[ \Gamma = \int_\R \gamma(x,y)\;dy =  \int_\R \gamma(z)\;dz,\]
and the nonlinear terms are given by

\begin{align}
 \label{eqn:nonlinear-P1} 
P^{(1)}(u,w) & =   \sigma(u)^{2} \sigma(w) ( 1-b \sigma(u) ),\\
 \label{eqn:nonlinear-P2} 
P^{(2)}(u,w)& =a - \sigma(u)^{2} \sigma(w)\\
 \label{eqn:nonlinear-Q1} 
 Q^{(1)}(u,w,v,z) &=
 \sigma(u) \sigma^{\prime}(u) \sigma(w) (2 -3 b\sigma(u)) v + \sigma(u)^{2} \sigma^{\prime}(w) (1- b\sigma(u)) z , \\ 
 \label{eqn:nonlinear-Q2}
 Q^{(2)}(u,w,v,z) &= 
 - 2 \sigma(u) \sigma^{\prime}(u) \sigma(w) v - \sigma(u)^{2} \sigma^{\prime}(w) z .
\end{align}

Here the function \( \sigma \in C^{\infty}(\R) \) is a cutoff function 
defined by the expression
\begin{equation}
\label{eqn:sigma}
\sigma(x) = 
\begin{cases}
-M & \text{if } x < - M, \\ 
x & \text{if } \abs{x} \leq M/2, \\ 
M & \text{if } x > M .
\end{cases}
\end{equation}
for some $M>0$, and satisfying  \( 0 \leq \sigma'(x) \leq 1 \) and  \( \abs{\sigma(x)} \leq M \) for all \( x \in \R \).

Equations \ref{eqn:model-aug} are supplemented by 
initial the initial conditions,
\begin{align*}
 u_0, \qquad v_{0} &= \partial_{x} u_{0} \quad \text{on } \Omega \times \{t=0\}, \\ 
 w_0, \qquad z_{0} &= \partial_{x} w_{0} \quad \text{on } \Omega \times \{t=0\},
\end{align*}
and boundary conditions/volume constraints 
\begin{align*}
 u& =0, \quad v = 0 \quad \text{in } \Omega^{c} \times [0,T), \\ 
w& =0, \quad z_{x} = 0 \quad \text{on } \partial\Omega \times [0,T).
\end{align*}
In particular, notice that we choose homogeneous nonlocal Dirichlet volume constraints and homogeneous Neumann boundary conditions for the \( v \) and \( z \) equations, respectively.

The derivation of system \eqref{eqn:model-aug} follows from composing the nonlinear terms in the nonlocal Klausmeier model \eqref{eqn:model} with the cutoff function $\sigma$, giving the first and second equations in the system. Then, differentiating these equations with respect to the spatial variable \( x \) and letting \( v = u_{x} \) and \( z = w_{x} \) gives us the third and fourth equations.
 In particular, we have that
\begin{align}
 \label{eqn:Ju-Ku-derivative}
 J u - \Gamma v &= \partial_{x} \left( K u \right) = \partial_{x} \int_{\R} \left( u(y)-u(x) \right) \gamma(x,y) \, dy   , \\
 \label{eqn:Q1-P1-derivative}
 Q^{(1)}(u,w,v,z) &= \partial_x P^{(1)}(u,v)=  \partial_{x} \left( \sigma(u)^{2} \sigma(w) (1- b \sigma(u)) \right) ,  \\
 \label{eqn:Q2-P2-derivative}
 Q^{(2)}(u,w,v,z) &= \partial_x P^{(2)} (u,v) = \partial_{x} \left( a - \sigma(u)^{2} \sigma(w) \right) .
\end{align}

Notice that  if both \( u \) and \( w \) are uniformly bounded by a positive constant \( C < M / 2 \), the first two equations in (\ref{eqn:model-aug}) reduce to our original Klausmeier model (\ref{eqn:model}). This follows from the definition of the function $\sigma$. The reason for adding this cutoff function is so that we can guarantee that the coefficients of the Galerkin approximations define a system of ODE with a Lipschitz continuous vector field. This will then allow us to obtain unique solutions and thus conclude that $v^m = \partial_x u^m$ and $z^m = \partial_x w^m$.

\subsection{Weak formulation for the augmented system}

We seek a time \( T > 0 \) and a tuple
\[
 (u,w,v,z) \in  L^{2}\left( 0,T, V \right)  \times L^{2}\left( 0,T, H^{1}_{0}(\Omega) \right) \times  L^{2}\left( 0,T, V \right) \times  L^{2}\left( 0,T, H^{1}(\Omega) \right)
\]
with 
\[
 (u_{t},w_{t},v_{t},z_{t}) \in \Big[ L^{2}\left( 0,T,V^{\prime} \right)  \times L^{2}\left( 0,T,H^{-1}(\Omega) \right) \Big]^{2} ,
\]
satisfying the system
\begin{equation}
\label{eqn:model-aug-weak}
\begin{aligned}
 \left< u_{t}, \varphi \right>_{(V^{\prime}\times V)} + d B[u,\varphi] + \mu \left( u, \varphi \right)_{L^{2}(\Omega)} &= \left( P^{(1)}(u,w) , \varphi \right)_{L^{2}(\Omega)}  \\ 
 \left< w_{t}, \psi \right>_{(H^{-1}(\Omega) \times H^{1}_{0}(\Omega))} + G[w, \psi] + \left( w, \psi \right)_{L^{2}(\Omega)} &= \left( P^{(2)}(u,w) , \psi \right)_{L^{2}(\Omega)} , \\
 \left< v_{t}, \varphi \right>_{(V^{\prime}\times V)} + ( \mu + d \Gamma ) \left(  v ,\varphi \right)_{L^{2}(\Omega)} &= d \left( J u, \varphi \right)_{L^{2}(\Omega)} + \left( Q^{(1)}(u,w,v,z), \varphi \right)_{L^2(\Omega)} \\ 
 \left< z_{t}, \rho \right>_{(H^{-1}(\Omega) \times H^{1}_{0}(\Omega))} + G[z,\rho] + \left( z, \rho \right)_{L^{2}(\Omega)} &= 
 \left(Q^{(2)}(u,w,v,z),  \rho \right)_{L^2(\Omega)},
\end{aligned}
\end{equation}
for each \( (\varphi, \psi, \rho) \in V \times H^{1}_{0}(\Omega) \times H^{1}(\Omega) \) and a.e. time \( 0\leq t < T \), with initial conditions 
\[
 u(x,0) = u_{0}(x), \quad 
 w(x,0) = w_{0}(x), \quad
 v(x,0) = \partial_{x} u_{0}(x), \quad 
 z(x,0) = \partial_{x} w_{0}(x), 
\]
for a.e. \( x \in \Omega \). 
As a starting point, we do not relate the functions \( v \) and \( z \) in any way to the spatial derivatives \( \partial_{x} u \) and \( \partial_{x} w \) of the functions \( u \) and \( w \). We later show, in Lemma \ref{lem:consequence-uniqueness}, that on a finite-dimensional subspace where we have \( u = u^{m} \) and \( w = w^{m} \) as linear combinations of basis functions, the solutions to (\ref{eqn:model-aug-weak}) satisfy \(  \partial_{x} u^{m} = v^{m} \) and \( \partial_{x} w^{m} = z^{m} \).

\subsection{Galerkin approximations and basis functions}

We define the Galerkin approximations for the variables \( u, w, v \) and \( z \). For a fixed positive integer \( m \), define the functions,
\begin{equation}
\label{eqn:galerkin-approxs}
\begin{aligned}
 u^{m} (x,t) &= \sum_{j=0}^{2m} d_{j}(t) \varphi_{j} (x), \quad 
 v^{m} (x,t) = \sum_{j=0}^{2m} \alpha_{j}(t) \varphi_{j}(x), \\
 w^{m} (x,t) &= \sum_{k=1}^{m} e_{k}(t) \psi_{k}(x), \quad 
 z^{m} (x,t) = \sum_{k=1}^{m} \beta_{k}(t) \rho_{k}(x) ,
\end{aligned}
\end{equation}
with coefficients  \( d_{j} (t) \), \( e_{k} (t) \), \( \alpha_{j} (t) \), \( \beta_{k}(t) \) in \( C^{1}([0,T]) \) satisfying,
\begin{equation}\label{eqn:intial_conditions}
\begin{aligned}
 & d_{j}(0) = \left( u_{0}, \varphi_{j} \right)_{L^{2}(\Omega)}, \quad 
\alpha_{j}(0) = \left( \partial_{x} u_{0}, \varphi_{j} \right)_{L^{2}(\Omega)}  \\ 
 & e_{k}(0) = \left( w_{0}, \psi_{k} \right)_{L^{2}(\Omega)},  \quad \beta_{k}(0) = \left( \partial_{x} w_{0}, \rho_{k} \right)_{L^{2}(\Omega)}, 
 \end{aligned}
\end{equation}
for $j = 0,\cdots, 2m$ and $k = 1, \cdots, m$. Here the basis functions  \( \left\{ \psi_{k} \right\} \subset H^1_0(\Omega) \),  \( \left\{ \rho_{k} \right\} \subset W \) (see equation (\ref{eqn:W}), and  \( \left\{ \varphi_{j} \right\}  \subset V \), are given by
\[ \psi_{k} (x) = \frac{1}{\sqrt{L}} \sin( k \pi x/L), \qquad \rho_{k} (x) = \frac{1}{\sqrt{L}} \sin( k \pi x/L), \]
while
\begin{equation}
\label{eqn:basis-plant}
 \varphi_{0}(x) = 1, \quad 
 \varphi_{j}(x) = 
 \begin{cases}
 \frac{1}{\sqrt{L}} \sin \left( \frac{(2j-1)\pi x}{2L} \right) , & \text{if } j \text{ is odd}, \\ 
 \frac{1}{\sqrt{L}} \cos \left( \frac{(2j-1) \pi x}{2L} \right) , & \text{if } j \text{ is even} ,
 \end{cases}
\end{equation}
with  $ j \in \N$ and  \( k \in \N \cup \{ 0 \} \).

Notice that since $\varphi_j(x) \in V$, we have that \( \varphi_{j}(x) = 0 \) for \( x \in \Omega^{c} \). Moreover, the elements in the set \( \left\{ \varphi_{j} \right\} \) do not all satisfy periodic, homogeneous Dirichlet, or homogeneous Neumann boundary conditions on \( \partial  \Omega \). As a result, solutions $u^m$ to equation \eqref{eqn:model-aug-weak} need not be continuous across the boundary. 
The following lemma shows that the non-standard set $\{ \varphi_j(x)\}$  forms an orthonormal basis for $V$.

\begin{lemma}
The elements $\varphi_j(x) \in C^\infty(\Omega)$ defined in expression \eqref{eqn:basis-plant} form an orthonormal basis for the space
\[ V= \{ u \in L^2(\R)\Big | \quad u =0, \quad \mbox{for} \quad x \in \Omega^c\}.\]
\end{lemma}

\begin{proof}
Notice that the set
\[
\{ \varphi_{j} \} = \{1\} \cup \left\{ \cos \left( \frac{j \pi x}{2L} \right), \sin\!\left( \frac{j\pi x}{2L} \right) \right\} ,
\]
for \( j \in \N\), is a basis for \( X = L^{2}\left( [-2L, 2L) \right) \) with periodic boundary conditions. Since \( V \subset X \), it follows that any element in $V$ can be described as a linear combination of elements in \( \{\varphi_{j}\} \).
However, if we restrict functions in this basis to the interval \( [-L,L] \), one finds that the inner product
\begin{align*}
\left( \cos\!\left( \frac{j\pi x}{2L}\right), \cos\!\left( \frac{k\pi x}{2L} \right) \right)_{L^{2}([-L,L])} &= \frac{2}{L\pi} \left( \frac{\sin\!\left( \frac{\pi}{2} (j-k) \right)}{j-k} + \frac{\sin\!\left(\frac{\pi}{2} (j+k) \right)}{j+k} \right),
\\
\left( \sin\!\left( \frac{j\pi x}{2L}\right), \sin\!\left( \frac{k\pi x}{2L} \right) \right)_{L^{2}([-L,L])} &= \frac{2}{L\pi} \left( \frac{\sin\!\left(\frac{\pi}{2} (j-k) \right)}{j-k} - \frac{\sin\!\left(\frac{\pi}{2} (j+k) \right)}{j+k} \right),
\end{align*}
for \(j, k \in \N\). 
These are zero if both \( j,k \) are even or odd and are nonzero if \( j \) and \( k \) differ in parity. This orthogonality relation means we can represent the even terms, \( \cos\!\left( \frac{2j\pi x}{2L} \right), \sin\!\left( \frac{2j\pi x}{2L} \right) \), using linear combinations of the odd terms, \( \cos\!\big( \frac{(2j-1)\pi x}{2L} \big), \sin\!\big( \frac{(2j-1)\pi x}{2L} \big) \), and vice versa. Therefore, we can discard either set of terms and retain a basis for \( L^{2}([-L, L]) \). We discard the even terms and retain the odd functions, as they do not satisfy local boundary conditions of any kind.

\end{proof}

\subsection{Projection onto finite subspace}
\label{sec:3_4}

Fixing $m$, inserting expressions \eqref{eqn:galerkin-approxs} into system \eqref{eqn:model-aug-weak}, 
and letting \( \varphi = \varphi_{k} \), \( \psi = \psi_{k}  \), and \( \rho = \rho_{k} \), 
 we arrive at the following system of o.d.e describing the time evolution of the coefficients   \( d_{j} (t) \), \( e_{j} (t) \), \( \alpha_{j} (t) \), \( \beta_{j}(t) \),
\begin{equation}
\label{eqn:model-ode-k}
\begin{aligned}
 \hat{I}_{k}^{j} d^{\prime}_{j}  & = - d \hat{B}_{k}^{j} d_{j} - \mu \hat{I}_{k}^{j} d_{j} + \tilde{P}^{(1)}_{k}(d_{j},e_{j}) ,  \\
 \hat{I}_{k}^{j} e^{\prime}_{j} &  = -  \hat{G}_{k}^{j} e_{j} - \hat{I}_{k}^{j} e_{j} + \tilde{P}^{(2)}_{k}(d_{j},e_{j}) ,  \\ 
 \hat{I}_{k}^{j} \alpha^{\prime}_{j}  & = - ( \mu + d \Gamma ) \hat{I}_{k}^{j} \alpha_{j} +  d \hat{J}_{k}^{j} d_{j} + \tilde{Q}^{(1)}_{k}(d_{j},e_{j},\alpha_{j}, \beta_{j}) , \\ 
 \hat{I}_{k}^{j} \beta^{\prime}_{j}  & = - \hat{M}_{k}^{j} \beta_{j} - \hat{I}_{k}^{j} \beta_{j} + \tilde{Q}^{(2)}_{k}(d_{j},e_{j}, \alpha_{j}, \beta_{j}) .
\end{aligned}
\end{equation}
Here we  use Einstein's notation, i.e., repeated indices represent summation in that index. 
The above expression can be derived using  the linearity of the inner product and of the bilinear forms \( B[\cdot,\cdot] \), \( G[\cdot,\cdot] \), together with  the fact that functions \( \left\{ \varphi_{k} \right\}_{k=0}^{2m} \), \( \left\{ \psi_{k} \right\}_{k=1}^{m} \), and \( \left\{ \rho_{k} \right\}_{k=0}^{m} \) are subsets of orthonormal bases for \( L^{2}(\Omega) \).
In terms of notation, $\hat{I}$ represents the identity matrix (of appropriate dimension), while the symbols \( \hat{B} \), \( \hat{G} \), $\hat{J}$, and \( \hat{M} \) denote the following matrices,
\begin{align}
 \label{eqn:bilinear-mat-B}
 \hat{B} &= 
( \hat{B}_{k}^{j} )_{0 \leq k,j \leq 2m} =
\left( B[\varphi_{j}, \varphi_{k}] \right)_{0 \leq k,j \leq 2m} =
\begin{bmatrix}
 B[\varphi_{0}, \varphi_{0}] & B[\varphi_{1}, \varphi_{0}] & \cdots & B[\varphi_{2m}, \varphi_{0}] \\ 
 B[\varphi_{0}, \varphi_{1}] & B[\varphi_{1}, \varphi_{1}] & \cdots & B[\varphi_{2m}, \varphi_{1}] \\
 \vdots & \vdots & \ddots & \vdots \\
 B[\varphi_{0}, \varphi_{2m}] & B[\varphi_{1}, \varphi_{2m}] & \cdots & B[\varphi_{2m}, \varphi_{2m}] 
\end{bmatrix} , \\
 \label{eqn:bilinear-mat-G}
 \hat{G} &= ( \hat{G}_{k}^{j} )_{1\leq k,j\leq m} = \left( G[\psi_{j}, \psi_{k}] \right)_{1 \leq k,j \leq m} 
  , \\ 
 \label{eqn:bilinear-mat-M}
 \hat{J} &= ( \hat{J}_{k}^{j} )_{0\leq k,j\leq m} = \left( J \varphi_j, \varphi_{k}] \right)_{0 \leq k,j \leq 2m}, \\
 \hat{M} &= ( \hat{M}_{k}^{j} )_{0\leq k,j\leq m} = \left( G[\rho_{j}, \rho_{k}] \right)_{0 \leq k,j \leq m}.
\end{align}

Finally, the expressions \( \tilde{P}^{(i)}_{k} \) and \( \tilde{Q}^{(i)}_{k} \), \( i=1,2 \) represent the projections of the functions \( P^{(1)} \) ,\( P^{(2)} \) ,\( Q^{(1)} \)  and \( Q^{(2)} \) defined in (\ref{eqn:nonlinear-P1}), (\ref{eqn:nonlinear-P2}), (\ref{eqn:nonlinear-Q1}), and (\ref{eqn:nonlinear-Q2}) onto the span of $ \varphi_k, \psi_k, \varphi_k, \rho_k$, respectively.

Notice that the right hand side of system \eqref{eqn:model-ode-k} defines a Lipschitz continuous vector field. Moreover, thanks to the cut-off function, $\sigma$, used in the definition of $P^{(i)}$ and $Q^{(i)}, i =1,2$,  this Lipschitz constant is independent of the parameter $m$. 
Invoking the Picard--Lindel{\"o}f theorem we therefore obtain the following Lemma
\begin{lemma}\label{lem:uniqueness}
Fix $m \in \N$. Then, there exists a time $T>0$, independent of $m$, such that the 
 system of ordinary differential equations \eqref{eqn:model-ode-k}
has a unique solution, \( d_{i} (t) \), \( e_{j} (t) \), \( \alpha_{i} (t) \), \( \beta_{j}(t) \) with $i =0,\cdots, 2m$ and $j = 1,\cdots,m$ 
defined on the time interval $ [0,T]$. Moreover, the solution is continuously differentiable on $[0,T]$.
\end{lemma}

With the uniqueness of the ODE system (\ref{eqn:model-ode-k}) in hand, we can prove that the Galerkin approximation \( u^{m} \) lies in \( H^{1}(\Omega) \).  To do this, we take advantage of the correspondence between the solutions to this ODE and the Galerkin approximations $u^m,w^m,v^m,z^m$, and show that the derivatives, \( \partial_{x} u^{m} \) and \( \partial_{x} w^{m} \), satisfy the  the weak formulation of the equations for \( v^m \) and \( z^m \), respectively. 
Using our uniqueness result we are then be able to conclude that \( \partial_{x} u^{m} = v^{m} \) and \( \partial_{x} w^{m} = z^{m} \). This is in essence the content of the Lemma \ref{lem:consequence-uniqueness}, shown next. 
In Section \ref{sec:4} we show
that the Galerkin approximations $u^m,w^m,v^m,z^m$ are bounded in $L^2(\Omega)$. It then follows that $u^m$ and $w^m$ are both  in $H^1(\Omega)$.

\begin{lemma}
\label{lem:consequence-uniqueness}
 Let  
\[
{\bf V}(t) = \left( d_{0} , \ldots, d_{2m} ,\, e_{1} , \ldots, e_{m}, \, \alpha_{0} , \ldots, \alpha_{2m}, \beta_{1}, \ldots, \beta_{m}  \right)(t) ,
\]
denote the unique solution to the system of ordinary differential equations (\ref{eqn:model-ode-k}) with corresponding initial conditions defined by \eqref{eqn:intial_conditions}.
Let \( u^{m}, v^{m}, w^{m}, z^{m} \) denote the Galerking approximations of the form (\ref{eqn:galerkin-approxs}), constructed using the vector ${\bf V}$ as coefficients. Then,
\[
 \partial_{x} u^{m} = v^{m} \quad \text{and} \quad 
 \partial_{x} w^{m} = z^{m}
\]
on \( \Omega \times [0, T] \).
\end{lemma}

\begin{proof}
Suppose \( u^{m}, w^{m} \) are of the form (\ref{eqn:galerkin-approxs}) and that their coefficients $\{d_k(t)\}_{k=0}^{2m}, \{ e_k(t)\}_{k =1}^m$ solve the first and second system of equation in (\ref{eqn:model-ode-k}). It then follows that their derivatives, $\partial_x u$ and $\partial_x w$, are functions in $V$ and $H^3(\Omega)$, respectively. We can therefore consider the expressions,
\begin{align}
\label{eqn:lem-plugin-v}
 F^{(1)}_k  = &\left( (\partial_{x} u^{m})_{t}, \varphi_{k} \right)_{L^{2}(\Omega)} + (\mu + d \Gamma) \left( \partial_{x} u^{m}, \varphi_{k} \right)_{L^{2}(\Omega)} - d \left( J u^{m}, \varphi_{k} \right)_{L^{2}(\Omega)} \\  \nonumber
 &- \left( Q^{(1)}(u^{m},w^{m},\partial_{x} u^{m}, \partial_{x} w^{m}), \varphi_{k} \right)_{L^{2}(\Omega)}  ,\\[2ex]
\label{eqn:lem-plugin-z}
F^{(2)}_k  =  &\left( (\partial_{x} w^{m})_{t}, \rho_{k} \right)_{L^{2}(\Omega)} + G[ \partial_{x} w^{m}, \rho_{k}]  + \left( \partial_{x} w^{m}, \rho_{k} \right)_{L^{2}(\Omega)} \\ \nonumber
 &- \left( Q^{(2)}(u^{m},w^{m},\partial_{x} u^{m}, \partial_{x} w^{m}), \rho_{k} \right)_{L^{2}(\Omega)},
\end{align}

where as before, the functions \( Q^{(i)} \), \( i=1,2 \) represent the nonlinear maps defined in (\ref{eqn:nonlinear-Q1}), (\ref{eqn:nonlinear-Q2}). In what follows we show that $F^{(i)}_k =0$ for $i =1,2$, since this is equivalent to $\partial_x u^m$ and $\partial_x w^m$ solving the third and fourth system of equations appearing in \eqref{eqn:model-ode-k}. Lemma \ref{lem:uniqueness} then implies that $v^m = \partial_x u^m$ and $z^m = \partial_x w^m$, giving us the desired result.

Notice first that the identities
\begin{align*}
 (\partial_{x} ( K u ), \varphi)_{L^2(\Omega)} &= ( J u - \Gamma \partial_{x} u, \varphi)_{L^2(\Omega)}  \\
 G[\partial_x w, \rho] &= - ( \partial_{xx} \left( \partial_{x} w \right) - \nu \partial_x w , \rho)_{L^2(\Omega)} 
 \end{align*}
are valid, since \( \partial_{x} u^{m} \) is in \( V \) and $w \in H^3(\Omega)$. 
Moreover, because  \( u, w \in C^{1}(\Omega) \) we have that 
\[ \partial_{x} P^{(i)} (u,w) = Q^{(i)}(u,w,\partial_{x} u, \partial_{x} w) , \qquad  i=1,2. \]
As a result, the expressions $F^{(i)}_k$ can be re-written as

\begin{align}
\label{eqn:lem-plugin-v2}
F^{(1)}_k = & \left( \partial_{x} (\partial_{t} u^{m}), \varphi_{k} \right)_{L^{2}(\Omega)}  - d \left( \partial_{x} (K u^{m}), \varphi_{k} \right)_{L^{2}(\Omega)} + \mu \left( \partial_{x} u^{m}, \varphi_{k} \right)_{L^{2}(\Omega)}  \\ \nonumber
& - \left( \partial_{x} P^{(1)}(u^{m},w^{m}), \varphi_{k} \right)_{L^{2}(\Omega)} ,\\[2ex]
\label{eqn:lem-plugin-z2}
F^{(2)}_k = &\left( \partial_{x} (\partial_{t} w^{m}), \rho_{k} \right)_{L^{2}(\Omega)}  -  (\partial_{xx} w^{m})_{x}  - \nu \partial_{x} w^{m}, \rho_{k} )_{L^2(\Omega)} + \left( \partial_{x} w^{m}, \rho_{k} \right)_{L^{2}(\Omega)}\\  \nonumber
& - \left( \partial_{x} P^{(2)}(u^{m},w^{m}), \rho_{k} \right)_{L^{2}(\Omega)}.
\end{align}
We can further simplify our notation by letting
\begin{align*}
N_{1} &= \partial_{t} u^{m} - d (K u^{m}) + \mu u^{m} - P^{(1)}(u^{m},w^{m}) , \\ 
N_{2} &= \partial_{t} w^{m} - (\partial_{xx} w + \nu \partial_{x} w) + w^{m} - P^{(2)}(u^{m},w^{m}) , 
\end{align*}
so that  (\ref{eqn:lem-plugin-v2}), (\ref{eqn:lem-plugin-z2}) are now given by
\begin{align*}
F^{(1)}_k = &\left( \partial_{x} N_{1}, \varphi_{k} \right)_{L^{2}(\Omega)} , \\ 
F^{(2)}_k =&\left( \partial_{x} N_{2}, \rho_{k} \right)_{L^{2}(\Omega)}.
\end{align*}

Next, from Lemma \ref{lem:formula-inner-product-bilinear} we know that the identity \( B[u,\varphi] = -\left( Ku, \varphi \right)_{L^{2}(\Omega)} \) holds for all \( u \in V \). Using this result, together with the definition of the bilinear form $G$,
we are then able to recognize
\begin{align*}
&\left(  N_{1}, \varphi_{k} \right)_{L^{2}(\Omega)} =0, \\ 
&\left( N_{2}, \rho_{k} \right)_{L^{2}(\Omega)} =0.
\end{align*}
as the first and second set of equations appearing in the system (\ref{eqn:model-ode-k}). 
Because \( u^{m} \) and $w^m$  solve these systems, the relations then show that the projections of \( N_{1} \) and $N_2$ onto the span of \( \left\{ \varphi_{k} \right\}_{k=0}^{2m} \) and $\{ \rho_k\}_{k=1}^{m}$ are equal to zero. 
At the same time, since \( \partial_{x} \varphi_{k} \in \left\{ \varphi_{k} \right\}_{k=0}^{2m}   \) and $ \partial_x \psi_k \in \{ \rho_k\}_{k=1}^{m}$,  it then follows that the functions defined by \( \partial_{x} N_{1} \) and \( \partial_{x} N_{2} \) are the zero functions in these finite subspaces. 
As a result we obtain that \( \left( \partial_{x} N_{1}, \varphi_{k} \right)_{L^{2}(\Omega)} = 0 \) for \( k = 0,1, \cdots, 2m \), and \( \left( \partial_{x} N_{2}, \rho_{k} \right)_{L^{2}(\Omega)} = 0 \) for \( k = 1,2, \cdots, m \). Consequently, the equations (\ref{eqn:lem-plugin-v}) (\ref{eqn:lem-plugin-z}) are indeed valid. 

\end{proof}

\section{Energy Estimates}
\label{sec:4}

Our goal in this section is to prove that the functions \( u^{m} \) and \( w^{m} \) that represent those elements in the sequences of approximations to solutions of (\ref{eqn:model-weak}) and are of the form (\ref{eqn:galerkin-approxs}), satisfy
\begin{align*}
 \norm{ u^{m} }_{ L^{\infty}(\Omega) } &< M / 2 , \\ 
 \norm{ w^{m} }_{ L^{\infty}(\Omega) } &< M / 2 ,
\end{align*}
for initial values that are sufficiently small in \( L^{2}(\Omega) \) and for sufficiently small times \( T > 0 \). (Throughout the remainder of the paper, whenever we write \( u^{m} \) or \( w^{m} \), we mean these to be as above.)

First, we prove energy estimates for the sequences of approximate solutions, 
\begin{equation*}
 \left\{ u^{m} \right\} , \, \left\{ w^{m} \right\}, \,  \left\{ (u_{x})^{m} \right\},\, \left\{ (w_{x})^{m} \right\} ,
\end{equation*}
defined as in \eqref{eqn:galerkin-approxs} and with coefficients satisfying the system of ODES \eqref{eqn:model-ode-k}. 
That is, we show that for a.e. $ t \in (0,T),$  these functions are in $L^2(\Omega)$. Then, since functions in \( H^{1}(\Omega) \) are embedded in \( L^{\infty}(\Omega) \) (see  \cite[Theorem 4.12 in Chapter 4]{adams2003sobolev}), we have that
\begin{align*}
 \norm{ u^{m} }_{ L^{\infty}(\Omega) } 
 &\leq C \norm{u^{m}}_{H^{1}(\Omega)} , \\
 \norm{ w^{m} }_{ L^{\infty}(\Omega) } 
 &\leq C \norm{w^{m}}_{H^{1}(\Omega)} , 
\end{align*}
for a.e. \( t \in [0,T] \), from which the desired result then follows.

For ease of exposition we re-write ODE system \eqref{eqn:model-ode-k} as,
\begin{equation}
\label{eqn:model-aug-weak-proj2}
\begin{aligned}
 \left< u^m_{t}, \varphi_k \right>_{(V^{\prime}\times V)} + d B[u^m,\varphi_k] + \mu \left( u^m, \varphi_k \right)_{L^{2}(\Omega)} &= \left( P^{(1)}(u^m,w^m) , \varphi_k \right)_{L^{2}(\Omega)}  \\ 
 \left< w^m_{t}, \psi_k \right>_{(H^{-1}(\Omega) \times H^{1}_{0}(\Omega))} + G[w^m, \psi_k] + \left( w^m, \psi_k \right)_{L^{2}(\Omega)} &= \left( P^{(2)}(u^m,w^m) , \psi_k \right)_{L^{2}(\Omega)} , \\
 \left< v^m_{t}, \varphi_k \right>_{(V^{\prime}\times V)} + ( \mu + d \Gamma ) \left(  v^m ,\varphi_k \right)_{L^{2}(\Omega)} &= d \left( J u^m, \varphi_k \right)_{L^{2}(\Omega)} \\
 &+ \left( Q^{(1)}(u^m,w^m,v^m,z^m), \varphi_k \right)_{L^2(\Omega)} \\ 
 \left< z^m_{t}, \rho_k \right>_{(H^{-1}(\Omega) \times H^{1}_{0}(\Omega))} + G[z^m,\rho_k] + \left( z^m, \rho_k \right)_{L^{2}(\Omega)} &= 
 \left(Q^{(2)}(u^m,w^m,v^m,z^m),  \rho_k \right)_{L^2(\Omega)}.
 \end{aligned}
\end{equation}
Notice that Lemma \ref{lem:consequence-uniqueness} is instrumental in relating the variables \( (v^{m} ,z^{m} ) \) appearing the last two equations in (\ref{eqn:model-aug-weak-proj2}) to the spatial derivatives of the variables \( (u^{m}, w^{m}) \) appearing in the first two equations.

Throughout this section, \( a, b, d, \mu, \nu >0 \) are model parameters, \( \Omega \) is the spatial domain, $T$
is the maximum existence time given by Lemma \ref{lem:uniqueness},
 \( \kappa, \kappa^{\prime} \chi, \chi^{\prime} > 0 \) are positive constants given in Lemmas \ref{lem:bilinear-form-water} and \ref{lem:bilinear-form-water-W}, and are associated to the coercivity of the bilinear form \( G[\cdot,\cdot] \) defined in (\ref{eqn:bilinear-form-water}).

\subsection{Estimates in $ L^{2}(\Omega) $}
\label{sub:L2estimates}

\begin{lemma}
\label{lem:u-w-energy-estimate}
Let \( m \in \N\), and assume that \( u^{m}, w^{m} \) are functions of the form (\ref{eqn:galerkin-approxs}) which represent solutions of the first two equations in (\ref{eqn:model-ode-k}), or equivalently of \eqref{eqn:model-aug-weak-proj2}.
Then, 
\begin{equation}
\label{eqn:u-w-energy-estimate}
 \norm{u^{m}(t)}_{L^{2}(\Omega)} + \norm{w^{m}(t)}_{L^{2}(\Omega)} 
 \leq 
 C_{1} \left( \norm{ u(0) }_{L^{2}(\Omega)} + \norm{ w(0) }_{L^{2}(\Omega)} \right) + C_{2}
\end{equation}
for all \( t \in [0,T] \), where \( C_{1} , C_{2}  >0 \) are the positive constants 
\begin{align}
 \label{eqn:C1}
 C_{1} &= \sqrt{2} \max\left\{ e^{\frac{1}{2}(M^{2}+bM^{3}) T}, e^{\chi T} \right\}, \\
 \label{eqn:C2}
 C_{2} &= a e^{\chi T} \sqrt{2 \abs{ \Omega } T} .
\end{align}

\end{lemma}
\begin{proof}
We begin by multiplying the first equation in (\ref{eqn:model-aug-weak-proj2}) by \( d_{k} \), \( k = 0, 1, \cdots, 2m \), and adding the equations together. Since \( d_{k} = d_{k}(t) \) is spatially independent, linearity of integration then yields the equation: 
\begin{equation}
 \label{eqn:u-01}
 \left( u^{m}_{t}, u^{m} \right)_{L^{2}(\Omega)} + B[u^{m},u^{m}] + \mu \left( u^{m},u^{m} \right)_{L^{2}(\Omega)} = \left( \sigma(u^{m})^{2} \sigma(w^{m}) (1-b\sigma(u^{m})), u^{m} \right)_{L^{2}(\Omega)}.
\end{equation}
We have \( \left( u^{m},u^{m} \right)_{L^{2}(\Omega)} = \norm{u^{m}}^{2}_{L^{2}(\Omega)} \).
Then, for ease of notation, we write \( \norm{\cdot} = \norm{\cdot}_{L^{2}(\Omega)} \), \( \left( \cdot,\cdot \right) = \left( \cdot,\cdot \right)_{L^{2}(\Omega)} \) and \( u = u^{m} \), \( w = w^{m} \).

Consider \( \left( u_{t}, u \right) = \int_{\Omega} u_{t}\,  u \, dx  \); since \( u = \sum_{k=0}^{2m} d_{k}(t) \varphi_{k}  \) and \( u_{t} = \sum_{k=0}^{2m} d_{k}^{\prime}(t) \varphi_{k} \), with \( \varphi_{k} \) either a \( \cos \) or \( \sin \) function (see equation (\ref{eqn:basis-plant})), with \( d_{k} \), \( d^{\prime}_{k} \) chosen to be continuous, we have \( \abs{ u (x,t) } \leq g_{1} \) and \( \abs{ u_{t} (x,t) } \leq g_{2}  \) for all \( (x,t) \in \Omega \times [0,T] \), for some \( g_{1}, g_{2} \in L^{1}(\Omega) \). Then, by the Dominated convergence theorem:
\begin{equation*}
 \frac{d}{dt} \int_{\Omega} u^{2} \, dx = \int_{\Omega} 2 u \, u_{t} \, dx = 2 \left( u_{t}, u \right).
\end{equation*}
Namely, \( \left( u_{t}, u \right) = \frac{1}{2} \norm{ u }^{2} \).

Then since \( d B[u,u] \geq 0 \) and \( \mu \norm{u}^{2} \geq 0 \), equation (\ref{eqn:u-01}) above implies the inequality, 
\begin{equation}
 \label{ineq:u-01}
 \frac{1}{2} \frac{d}{dt} \norm{u(t)}^{2} \leq \left( \sigma(u^{m})^{2} \sigma(w^{m}) (1-b\sigma(u^{m})), u^{m} \right).
\end{equation}

Because \( z \leq \abs{z} \) for all \( z \in \R \), we bound the modulus of the right-hand side above using the triangle inequality: 
\begin{align*}
 \abs{ \left( \sigma(u)^{2} \sigma(w) (1-b\sigma(u)), u \right) } &\leq \abs{ \left( \sigma(u)^{2} \sigma(w), u \right) } + b \abs{ \left( \sigma(u)^{3} \sigma(w), u \right) } \\ 
 &=  \int_{\Omega} \abs{ \sigma(u)^{2} \sigma(w) u } \, dx   + b \int_{\Omega} \abs{ \sigma(u)^{3} \sigma(w) u } \, dx \\ 
 &\leq M^{2} \int_{\Omega} \abs{ \sigma(u) u }  \, dx  + b M^{3} \int_{\Omega} \abs{ \sigma(u) u } \, dx .
\end{align*}
Here, we used the property that \( \sigma(\cdot) \) is a cutoff function with upper bound the fixed value \( M \) (see equation (\ref{eqn:sigma})).

Now recall that for every \( x \in \R \), \( \abs{ \sigma(u(x)) } \leq \abs{ u(x) } \). Hence, 
\[
 \int_{\Omega} \abs{ \sigma(u) u } \, dx  \leq \int_{\Omega} \abs{ u^{2} } \, dx = \norm{u}^{2}.
\]

Thus, combining all of the above, the differential inequality (\ref{ineq:u-01}) gives us the bound, 
\begin{equation*}
 \frac{d}{dt} \norm{u(t)}^{2} \leq 2 \left( M^{2} + b M^{3} \right) \norm{u(t)}^{2} .
\end{equation*}
Gronwall's inequality then gives us the following estimate:
\begin{equation}
\label{eqn:u-energy-estimate}
\norm{u(t)}^{2} \leq e^{2 (M^{2}+bM^{3}) t} \norm{ u(0) }^{2},
\end{equation}
for all \( t \in [0,T] \).

Next, we operate on the second equation in (\ref{eqn:model-aug-weak-proj2}) similarly as we have done for the first.
Namely, multiply this equation by \( e_{k} \), \( k = 1, 2 , \cdots , m \), and add them together. 
This yields:
\begin{equation*}
 \left( w_{t}, w \right) + G[w,w] + \left( w,w \right) = \left( a - \sigma(u)^{2} \sigma(w), w \right)
\end{equation*}
We rewrite this as 
\begin{equation*}
\left( w_{t}, w \right) + G[w,w] = \left( a - w, w \right) - \left( \sigma(u)^{2} \sigma(w), w \right) .
\end{equation*}
Notice that \( \left( \sigma(u)^{2} \sigma(w), w \right) \geq 0 \) because \( \sigma(w(x)) \) and \( w(x) \) are of the same sign for all \( x \in \R \) and the other factor under the integration, \( \sigma(u)^{2} \), is also nonnegative. This implies 
\begin{equation*}
 \left( w_{t}, w \right) + G[w,w] \leq 
 \left( a - w, w \right) .
\end{equation*}
Then, by Lemma \ref{lem:bilinear-form-water}, there exist constants \( \kappa, \chi > 0 \), satisfying
\begin{equation*}
 \kappa \norm{ w }_{H^{1}_{0}(\Omega)}^{2} \leq G[w,w] + \chi \norm{ w }^{2}.
\end{equation*}
Hence, 
\begin{equation}
 \label{ineq:w-01}
 \frac{1}{2} \frac{d}{dt} \norm{ w(t) }^{2} + \kappa \norm{ w }_{H^{1}_{0}(\Omega)}^{2} - \chi \norm{ w }^{2}\leq 
 \left( a - w, w \right) 
\end{equation}
holds, and since \(  \kappa \norm{ w }_{H^{1}_{0}(\Omega)}^{2} \geq 0\), the estimate (\ref{ineq:w-01}) implies  
\begin{equation}
 \label{ineq:w-02}
  \frac{1}{2} \frac{d}{dt} \norm{ w(t) }^{2} \leq \left( a - w, w \right) + \chi \norm{ w }^{2}.
\end{equation}
Then, we compute: 
\begin{align*}
 2 \left( a - w, w \right) &= \int_{\Omega} 2 a w - 2 w^{2} \, dx \\ 
 &= \int_{\Omega} - \left( w^{2} - 2 a w \right) - w^{2} \, dx.
\end{align*}
Completing the square inside the parenthesis gives: 
\begin{equation*}
\int_{\Omega} - \left( w^{2} - 2 a w + a^{2} \right) + a^{2} - w^{2}  \, dx = \int_{\Omega} a^{2} - (a-w)^{2} - w^{2} \, dx.
\end{equation*}
These operations then give us the formula: 
\begin{equation*}
 \left( a - w, w \right) = \frac{1}{2} \left( a^{2} \abs{ \Omega } - \norm{ a - w }^{2} - \norm{ w }^{2} \right) .
\end{equation*}
Since \( \norm{ a - w }^{2}, \norm{ w }^{2} \geq 0 \) are nonnegative quantities, when we return to equation (\ref{ineq:w-02}), namely, \( \frac{1}{2} \frac{d}{dt} \norm{ w(t) }^{2} \leq \left( a - w, w \right) + \chi \norm{ w }^{2} \), the formula above implies that 
\begin{equation}
\label{ineq:w-03}
 \frac{d}{dt} \norm{ w(t) }^{2} \leq 2 \chi \norm{ w }^{2} + a^{2} \abs{ \Omega } .
\end{equation}
From Gronwall's inequality:
\begin{equation}
\label{eqn:w-energy-estimate}
\norm{w(t)}^{2} \leq e^{2 \chi t} \left( \norm{w(0)}^{2} + a^{2} \abs{ \Omega } t \right) ,
\end{equation}
for all \( t \in [0,T] \).

Then, since \( f_{1}(t) = \eps t \) and \( f_{2}(t) = e^{\omega t} \), \( \eps, \omega \in \R_{\geq 0} = \left\{ x \in \R \mid x \geq 0 \right\} \) are increasing functions of \( t \), it follows that 
\begin{align}
\label{eqn:u-energy-estimate2}
 \max_{0 \leq t \leq T} \norm{ u(t) }^{2} &\leq e^{2 (M^{2} + bM^{3}) T} \norm{ u(0) }^{2}, \\ 
\label{eqn:w-energy-estimate2}
 \max_{0 \leq t \leq T} \norm{ w(t) }^{2} &\leq e^{2 \chi T} \left( \norm{ w(0) }^{2} + a^{2} \abs{ \Omega } T  \right), 
\end{align}
both hold. 
We use this conclusion to write \( T \) rather than \( t \) in applying Gronwall's inequality in future estimates where it applies.

We now consider the sum of the estimates (\ref{eqn:u-energy-estimate2}) and (\ref{eqn:w-energy-estimate2}), take the square root, and bound the corresponding right-hand side. 
Let \( \tilde{ C_{1} } = \max\left\{ e^{(M^{2}+bM^{3}) T}, e^{2\chi T} \right\} \) and \( \tilde{C}_{2} = a^{2} \abs{ \Omega } T e^{2 \chi T}  \).
Then, after adding the above inequalities together, we have 
\begin{equation*}
 \norm{ u(t) }^{2} + \norm{ w(t) }^{2} \leq \tilde{C_{1}} \left( \norm{ u(0) }^{2} + \norm{ w(0) }^{2} \right) + \tilde{C}_{2} .
\end{equation*}
Because of the inequality,
\begin{equation}
\label{ineq:sqrt}
\abs{ a } + \abs{ b } \leq \sqrt{2} \sqrt{ a^{2} + b^{2} } \leq \sqrt{2} \left( \abs{a} + \abs{b} \right) ,
\end{equation}
which holds for real numbers \( a \) and \( b \), and because the square-root function is subadditive, we have
\begin{align*}
  \norm{ u(t) } + \norm{ w(t) }  &\leq \sqrt{2} \sqrt{ \norm{ u(t) }^{2} + \norm{ w(t) }^{2} } \\
&\leq \sqrt{ 2 \tilde{ C_{1} } } \sqrt{ \norm{ u(0) }^{2} + \norm{ w(0) }^{2} } + \sqrt{\tilde{C}_{2}} \\ 
&\leq \sqrt{ 2 \tilde{ C_{1} } } \left( \norm{ u(0) } + \norm{ w(0) }  \right) + \sqrt{\tilde{C}_{2}} .
\end{align*}
Then,
\begin{equation*}
\begin{aligned}
 \sqrt{ \tilde{C_{1}} } &= \sqrt{ \max\left\{ e^{(M^{2}+bM^{3}) T}, e^{2\chi T} \right\} } = \max\left\{ e^{\frac{1}{2}(M^{2}+bM^{3}) T}, e^{\chi T} \right\}, \\
 \sqrt{\tilde{C}_{2}} &= a e^{\chi T} \sqrt{ 2 \abs{\Omega} T } ,
\end{aligned}
\end{equation*}
and the estimate above, 
\begin{equation*}
 \norm{ u(t) } + \norm{ w(t) } \leq C_{1} \left( \norm{ u(0) } + \norm{ w(0) }  \right) + C_{2}, 
\end{equation*}
is (\ref{eqn:u-w-energy-estimate}) after letting \( C_{1} = \sqrt{2 \tilde{C_{1}}} \) and \( C_{2} = \sqrt{\tilde{C}_{2}} \), as in (\ref{eqn:C1}) and (\ref{eqn:C2}), respectively. 

\end{proof}

Next, we prove a similar estimates for \( u^{m}_{x} \) and \( w^{m}_{x} \) in the third and fourth equations in (\ref{eqn:model-aug-weak-proj2}).

\begin{lemma}
\label{lem:ux-wx-energy-estimate}
Let \( m \in \N\), and assume that \( u_{x}^{m}, w_{x}^{m} \) are functions of the form 
(\ref{eqn:galerkin-approxs}) which represent solutions to
the last two equations in  \eqref{eqn:model-ode-k}, or equivalently of (\ref{eqn:model-aug-weak-proj2}). 
Then, 
\begin{equation}
\label{eqn:ux-wx-energy-estimate}
 \norm{u_{x}^{m}(t)}_{L^{2}(\Omega)} + \norm{w_{x}^{m}(t)}_{L^{2}(\Omega)} \leq
 D_{1} \left( \norm{u_{x}(0)}_{L^{2}(\Omega)} + \norm{w_{x}(0)}_{L^{2}(\Omega)} \right) + D_{2}
\end{equation}
for all \( t \in [0,T] \), where \( D_{1}, D_{2} > 0 \) are the positive constants, 
\begin{align}
 \label{eqn:D1}
 D_{1} &= \sqrt{2} e^{\tilde{D}_{1} T}, \\
 \label{eqn:D2}
 D_{2} &= D_{1} \sqrt{2 \tilde{D_{2}} T} = 2 e^{\tilde{D}_{1} T} \sqrt{\tilde{D}_{2} T } .
\end{align}
Here, 
\begin{align}
 \label{eqn:tilde-D1}
 \tilde{D}_{1} &= 2 \max\left\{ 2M + 3bM^{3}, \frac{M^{2}+bM^{3}}{2} \right\} + 2 \max\left\{ M^{2}, 3bM^{3} \right\} + \max\left\{ \frac{d \norm{J}_{\text{op}}}{2}, \chi^{\prime} \right\} , \\
 \label{eqn:tilde-D2}
 \tilde{D}_{2} &= \frac{d \norm{J}_{\text{op}}}{2} e^{2 (M^{2}+bM^{3}) T} \norm{ u(0) }_{L^{2}(\Omega)}^{2}.
\end{align}
\end{lemma}
\begin{proof}
For ease of notation, we do not write the superscript \( m \), and put \( u_{x}^{m} = u_{x} \) and \( w_{x}^{m} = w_{x} \). We also write \( \norm{ \cdot } = \norm{ \cdot }_{L^{2}(\Omega)} \) and \( \left( \cdot, \cdot \right) = \left( \cdot, \cdot \right)_{L^{2}(\Omega)} \).

Multiply the fourth equation in (\ref{eqn:model-aug-weak-proj2}) by \( \beta_{k} \), \( k= 1, 2 ,\cdots, m \), and add the equations together. 
We obtain
\begin{equation}
\label{eqn:2}
\frac{1}{2} \frac{d}{dt} \norm{ w_{x}(t) }^{2} + G[w_{x},w_{x}] + \norm{ w_{x} }^{2} = \left( Q^{(2)}(u,w), w_{x} \right) .
\end{equation}

By Lemma \ref{lem:bilinear-form-water-W}, there exist positive constants \( \kappa^{\prime}, \chi^{\prime} > 0 \), satisfying 
\begin{equation}
 \kappa^{\prime} \norm{ w_{x} }_{W}^{2} \leq G[w_{x},w_{x}] + \chi^{\prime} \norm{ w_{x} }_{L^{2}(\Omega)}^{2}
\end{equation}
Hence, (\ref{eqn:2}) implies the inequality, 
\begin{equation}
\label{ineq:2-01}
\frac{1}{2} \frac{d}{dt} \norm{ w_{x}(t) }^{2} +  \kappa^{\prime} \norm{ w_{x} }_{W}^{2} - \chi^{\prime} \norm{ w_{x} }^{2} + \norm{ w_{x} }^{2} \leq \left( Q^{(2)}(u,w), w_{x} \right) .
\end{equation}
Moreover, since \( \kappa^{\prime} \norm{ w_{x} }_{W}^{2}, \norm{ w_{x} }_{L^{2}(\Omega)}^{2} \geq 0 \), inequality (\ref{ineq:2-01}) implies 
\begin{equation}
\label{ineq:2-02}
\frac{1}{2} \frac{d}{dt} \norm{ w_{x}(t) } \leq 
\chi^{\prime} \norm{ w_{x} }^{2}+ \left( Q^{(2)}(u,w), w_{x} \right) .
\end{equation}
Next, we bound 
\begin{equation}
\label{eqn:Q2-term}
\begin{aligned}
\left( Q^{(2)}(u,w), w_{x} \right) &= - 2 \left( \sigma(u) \sigma^{\prime}(u) \sigma(w) u_{x}, w_{x}  \right) \\ 
&\quad - 3 b \left( \sigma(u)^{2}\sigma^{\prime}(u) \sigma(w) w_{x}, w_{x} \right) .
\end{aligned}
\end{equation}
Using the defining property of the cutoff function \( \sigma \), given in (\ref{eqn:sigma}), we have
\begin{align*}
 \abs{ \left( Q^{(2)}(u,w), w_{x} \right) } &\leq 2 M^{2} \abs{ \left( u_{x},w_{x} \right) } + 3 b M^{3} \abs{ \left( w_{x}, w_{x} \right)  } \\ 
 &\leq 2 M^{2} \norm{ u_{x} } \norm{ w_{x} } + 3 b M^{3} \norm{w_{x} }^{2} . 
\end{align*}
where we also applied the Cauchy-Schwarz inequality in the second line.
Then, by Cauchy's inequality (\( ab \leq \frac{1}{2} ( a^{2} + b^{2}) \), \( a,b\in\R \)), we conclude: 
\begin{align*}
 \abs{ \left( Q^{(2)}(u,w), w_{x}  \right) } &\leq 
 M^{2} \left( \norm{ u_{x} }^{2} + \norm{ w_{x} } \right) + 3 b M^{3} \norm{w_{x} } \\ 
 &\leq 2 \max\left\{ M^{2}, 3 b M^{3} \right\} \left( \norm{ u_{x} }^{2} + \norm{ w_{x} }^{2} \right) .
\end{align*}
From the above, we deduce that (\ref{ineq:2-02}) implies
\begin{equation}
\label{ineq:2-03}
 \frac{1}{2} \frac{d}{dt} \norm{ w_{x}(t) }^{2} \leq 
\chi \norm{ w_{x} }^{2}+ 2 \max\left\{ M^{2}, 3 b M^{3} \right\} \left( \norm{ u_{x} }^{2} + \norm{ w_{x} }^{2} \right) .
\end{equation}

Now multiply the third equation in (\ref{eqn:model-aug-weak-proj2}) by \( \alpha_{k} \), \( k=0,1,\cdots, 2m \), and add the equations together. This yields:
\begin{equation*}
 \frac{1}{2} \frac{d}{dt} \norm{ u_{x}(t) }^{2} + ( \mu + d \Gamma) \norm{ u_{x} }^{2} = d \left( Ju, u_{x} \right) + \left( Q^{(1)}(u,w), u_{x} \right)
\end{equation*}
Since \( (\mu+d \Gamma) \norm{ u_{x} }^{2} \geq 0 \), we have 
\begin{equation}
\label{ineq:1-01}
 \frac{1}{2} \frac{d}{dt} \norm{ u_{x}(t) }^{2} \leq d \left( Ju, u_{x} \right) + \left( Q^{(1)}(u,w), u_{x} \right) .
\end{equation}
Recall that because of Hypothesis \ref{hyp:kernel}, we have that the operator \( J : V\to L^{2}(\Omega) \), defined as  
\begin{equation*}
 \left( Ju \right)(x) = \begin{cases}
 \int_{\R} \left( u(y) - u(x) \right) \partial_{x} \gamma (x,y) \, dy & \text{if } x \in \Omega, \\ 
 0 & \text{if } x\not\in \Omega ,
 \end{cases} 
\end{equation*}
is a bounded linear operator. Thus, by Cauchy-Schwarz,
\begin{equation*}
 \left( Ju, u_{x} \right) \leq 
 \norm{ Ju }  \norm{ u_{x} } \leq 
 \norm{J}_{\text{op}} \norm{u} \norm{u_{x}} .
\end{equation*}
Here, \( \norm{J}_{\text{op}} = \sup_{\norm{u}=1} \norm{Ju} \) is the operator norm of \( J \).
Next, Cauchy's inequality gives us,
\begin{equation*}
 \left( Ju, u_{x} \right) \leq
 \frac{\norm{J}_{\text{op}}}{2} \left( \norm{u}^{2} + \norm{u_{x}}^{2} \right) .
\end{equation*}
In the proof of the preceding lemma, we proved the following energy estimate:
\begin{equation*}
 \max_{0 \leq t \leq T} \norm{u(t)}^{2} \leq e^{2 (M^{2}+bM^{3}) T} \norm{ u(0) }^{2}
\end{equation*}
(see (\ref{eqn:u-energy-estimate})). 
Hence, 
\begin{equation*}
 \left( Ju, u_{x} \right) \leq 
 \frac{\norm{J}_{\text{op}}}{2} \left( e^{2 (M^{2}+bM^{3}) T} \norm{u(0)}^{2} + \norm{u_{x}}^{2} \right) .
\end{equation*}
Then, using the above, (\ref{ineq:1-01}), that is, 
\begin{equation*}
\frac{1}{2} \frac{d}{dt} \norm{ u_{x}(t) }^{2} \leq d \left( Ju, u_{x} \right) + \left( Q^{(1)}(u,w), u_{x} \right) , 
\end{equation*}
implies:
\begin{equation}
\label{ineq:1-02}
\begin{aligned}
 \frac{1}{2} \frac{d}{dt} \norm{ u_{x}(t) }^{2} &\leq 
 \frac{d \norm{J}_{\text{op}}}{2}  \norm{u_{x}}^{2} + \left( Q^{(1)}(u,w), u_{x} \right) \\ 
&\qquad + \frac{d \norm{J}_{\text{op}}}{2} e^{2 (M^{2}+bM^{3}) T} \norm{ u(0) }^{2}  .
\end{aligned}
\end{equation}
We now consider,
\begin{equation}
\label{eqn:Q1-term}
\begin{aligned}
\left( Q^{(1)}(u,w), u_{x} \right) &= \left( \sigma(u) \sigma^{\prime}(u) \sigma(w) (2-3 b \sigma(u)) u_{x}, u_{x} \right) \\ 
&\quad + \left( \sigma(u)^{2} \sigma^{\prime}(w) (1- b\sigma(u) ) w_{x}, u_{x} \right) .
\end{aligned}
\end{equation}
We have 
\begin{equation*}
\abs{ \left( Q^{(1)}(u,w), u_{x} \right) } \leq 
 (2 M^{2} + 3 b M^{3}) \abs{ \left( u_{x}, u_{x} \right) } + (M^{2}+ b M^{3}) \abs{ \left( w_{x},u_{x} \right) },
\end{equation*}
using the definition of \( \sigma \) (see (\ref{eqn:sigma})).
By Cauchy-Schwarz, we have \( \abs{ \left( w_{x},u_{x} \right) } \leq \norm{ u_{x} } \norm{ w_{x} } \), and after applying Cauchy's inequality, we find that
\begin{align*}
\abs{ \left( Q^{(1)}(u,w), u_{x} \right) } &\leq (2 M + 3 b M^{3}) \norm{ u_{x} }^{2} + \frac{(M^{2}+ b M^{3})}{2} \left( \norm{ u_{x} }^{2} + \norm{ w_{x} }^{2} \right)  \\ 
&\leq 2 \max\left\{ 2 M^{2} + 3 b M^{3}, \frac{M^{2}+ b M^{3}}{2}  \right\} \left( \norm{ u_{x} }^{2} + \norm{ w_{x} }^{2} \right)
\end{align*}
Then, inequality (\ref{ineq:1-02})
gives us:
\begin{equation}
\label{ineq:1-03}
\begin{aligned}
 \frac{1}{2} \frac{d}{dt} \norm{ u_{x}(t) }^{2} &\leq 
 \frac{d \norm{J}_{\text{op}}}{2}  \norm{u_{x}}^{2} + \frac{d \norm{J}_{\text{op}}}{2} e^{2 (M^{2}+bM^{3}) T} \norm{ u(0) }^{2} \\ 
 &\qquad + 2 \max\left\{ 2 M^{2} + 3 b M^{3}, \frac{M^{2}+ b M^{3}}{2}  \right\} \left( \norm{ u_{x} }^{2} + \norm{ w_{x} }^{2} \right)
\end{aligned}
\end{equation}
Earlier, we arrived at the differential inequality (\ref{ineq:2-03}), reproduced here,  
\begin{equation*}
 \frac{1}{2} \frac{d}{dt} \norm{ w_{x}(t) }^{2} \leq 
\chi^{\prime} \norm{ w_{x} }^{2}+ 2 \max\left\{ M^{2}, 3 b M^{3} \right\} \left( \norm{ u_{x} }^{2} + \norm{ w_{x} }^{2} \right) .
\end{equation*}
Adding inequalities (\ref{ineq:1-03}) and (\ref{ineq:2-03}) together then yields the equation:
\begin{align*}
 \frac{1}{2} \frac{d}{dt} \left( \norm{u_{x}(t)}^{2} + \norm{w_{x}(t)}^{2} \right) &\leq 
 \tilde{D}_{1} \left( \norm{ u_{x} }^{2} + \norm{ w_{x} }^{2} \right) + \tilde{D}_{2} ,
\end{align*}
where 
\begin{equation}
\label{eqn:tilde-D1-D2}
\begin{aligned}
\tilde{D}_{1} &= 2 \max\left\{ 2M + 3bM^{3}, \frac{M^{2}+bM^{3}}{2} \right\} + 2 \max\left\{ M^{2}, 3bM^{3} \right\} + \max\left\{ \frac{d \norm{J}_{\text{op}}}{2}, \chi^{\prime} \right\} , \\
\tilde{D}_{2} &= \frac{d \norm{J}_{\text{op}}}{2} e^{2 (M^{2}+bM^{3}) T} \norm{ u(0) }^{2}.
\end{aligned}
\end{equation}
Letting \( \eta(t) = \norm{u_{x}(t)}^{2} + \norm{w_{x}(t)}^{2} \), the above reads
\begin{equation}
\label{ineq:eta}
\eta^{\prime}(t) \leq 2 \tilde{D}_{1} \eta(t) + 2 \tilde{D}_{2} .
\end{equation}
By Gronwall's inequality, (\ref{ineq:eta}) then yields:
\begin{equation}
\label{ineq:eta-Gronwall}
\eta(t) \leq e^{2 \tilde{D}_{1} T} \left( \eta(0) + 2 \tilde{D}_{2} T \right) .
\end{equation}
for all \( t \in [0,T] \). That is, 
\begin{equation*}
 \norm{u_{x}(t)}^{2} + \norm{w_{x}(t)}^{2} \leq 
 e^{2 \tilde{D}_{1} T} \left( \norm{u_{x}(0)}^{2} + \norm{w_{x}(0)}^{2} + 2 \tilde{D}_{2} T \right)
\end{equation*}

Then, after applying inequality (\ref{ineq:sqrt}) (involving square roots), we obtain 
\begin{equation}
\label{ineq:tgt-01}
\begin{aligned}
 \norm{ u_{x}(t) } + \norm{ w_{x}(t) } &\leq \sqrt{2} \sqrt{ \norm{u_{x}(t)}^{2} + \norm{w_{x}(t)}^{2} } \\ 
 &\leq \sqrt{2} e^{\tilde{D}_{1} T} \left( \sqrt{ \norm{u_{x}(0)}^{2} + \norm{w_{x}(0)}^{2} } + \sqrt{2 \tilde{D}_{2} T }  \right) \\ 
 &\leq \sqrt{2} e^{\tilde{D}_{1} T} \left( \norm{u_{x}(0)} + \norm{w_{x}(0)} + \sqrt{2 \tilde{D}_{2} T } \right) .
\end{aligned}
\end{equation}
Therefore, (\ref{ineq:tgt-01}) implies 
\begin{equation*}
\norm{ u_{x}(t) } + \norm{ w_{x}(t) } \leq 
D_{1} \left( \norm{u_{x}(0)} + \norm{w_{x}(0)} \right) + D_{2},
\end{equation*}
where 
\begin{align*}
D_{1} &= \sqrt{2} e^{\tilde{D}_{1} T}, \\ 
D_{2} &= D_{1} \sqrt{2 \tilde{D_{2}} T} = 2 e^{\tilde{D}_{1} T} \sqrt{\tilde{D}_{2} T } ,
\end{align*}
and 
\begin{align*}
\tilde{D}_{1} &= 2 \max\left\{ 2M + 3bM^{3}, \frac{M^{2}+bM^{3}}{2} \right\} + 2 \max\left\{ M^{2}, 3bM^{3} \right\} + \max\left\{ \frac{d \norm{J}_{\text{op}}}{2}, \chi^{\prime} \right\} , \\
\tilde{D}_{2} &= \frac{d \norm{J}_{\text{op}}}{2} e^{2 (M^{2}+bM^{3}) T} \norm{ u(0) }^{2}.
\end{align*}
\end{proof}

\begin{remark}
\label{rem:u-w-H1}
Recall the main result of the previous section: We proved that the sequences of approximate solutions, \( \left\{ u^{m} \right\} \) and \( \left\{ w^{m} \right\} \), to the first and second equations in (\ref{eqn:model-aug-weak-proj2}), satisfy:
\begin{equation*}
 \partial_{x} u^{m} = v^{m} \quad 
 \text{and} \quad 
 \partial_{x} w^{m} = z^{m} 
\end{equation*}
on  \( \Omega \times [0,T] \), where \( v^{m} = u_{x}^{m} \) and \( z^{m} = w_{x}^{m} \) solve the third and fourth equations in that system, (\ref{eqn:model-aug-weak-proj2}).
Then, together with the conclusions of Lemmas \ref{lem:u-w-energy-estimate} and \ref{lem:ux-wx-energy-estimate}, we conclude that both of the approximate solutions \(  u^{m} \), \( w^{m} \) lie in the Sobolev space \( H^{1}(\Omega) \) (in particular, \( w^{m} \in H^{1}_{0}(\Omega) \)). 
Therefore, we identify the functions \( u^{m}, w^{m} \), with the continuous functions \( (u^{m})^{*}, (w^{m})^{*} \), satisfying, 
\begin{equation*}
 u^{m}(x,t) = (u^{m})^{*}(x,t) 
 \quad \text{and} \quad 
 w^{m}(x,t) = (w^{m})^{*}(x,t) ,
\end{equation*}
for a.e. \( x \in \Omega \), for all \( t \in [0,T] \). 
\end{remark}

\subsection{Estimates in $ H^{1}(\Omega) $}

In the following, we combine the two preceding lemmas to provide an estimate for
\begin{equation*}
 \norm{u}_{H^{1}(\Omega)} + \norm{w}_{H^{1}_{0}(\Omega)} .
\end{equation*}
\begin{lemma}
\label{lem:u-w-energy-estimate-H1}
Let \( m \in \N\), and assume that \( u^{m}, w^{m} \) are functions of the form (\ref{eqn:galerkin-approxs}) which represent elements belonging to sequences of approximations to solutions of the first two equations in (\ref{eqn:model-aug-weak-proj2}). 
Then, there exist constants \( E_{1}, E_{2} \), and \( E_{3} \), satisfying, 
\begin{equation}
\label{eqn:u-w-H1}
\begin{aligned}
 \norm{ u^{m} }_{H^{1}(\Omega)} + \norm{ w^{m} }_{H^{1}_{0}(\Omega)} &\leq 
 E_{1} \left( \norm{ u(0) }_{L^{2}(\Omega)} + \norm{ w(0) }_{L^{2}(\Omega)} \right) \\ 
 &\quad + E_{2} \left( \norm{ u_{x}(0) }_{L^{2}(\Omega)} + \norm{ w_{x}(0) }_{L^{2}(\Omega)} \right) + E_{3}
\end{aligned}
\end{equation}
where 
\begin{align}
 \label{eqn:E1}
 E_{1} &= 2 \max\left\{ C_{1}, D_{1} e^{(M^{2}+bM^{3}) T} \sqrt{ d \norm{J}_{\text{op}} } \right\}, \\ 
 \label{eqn:E2}
 E_{2} &= D_{1} = \sqrt{2} e^{\tilde{D}_{1} T}, \\ 
 \label{eqn:E3}
 E_{3} &= C_{2} = a e^{\chi T} \sqrt{2 \abs{ \Omega } T}  .
\end{align}
Here, the constants \( C_{1}, C_{2} \) are defined in Lemma \ref{lem:u-w-energy-estimate}; \( D_{1} \), \( \tilde{D}_{1} \) are defined in Lemma \ref{lem:ux-wx-energy-estimate}. Namely, they are defined in equations (\ref{eqn:C1}), (\ref{eqn:C2}), (\ref{eqn:D1}), and (\ref{eqn:tilde-D1}), respectively.
\end{lemma}
\begin{proof}
As before, we drop the superscript \( m \), and write \( u = u^{m} \), \( w = w^{m} \), \( u_{x} = u_{x}^{m} \), and \( w_{x} = w_{x}^{m} \). We also write \( \norm{ \cdot } = \norm{ \cdot }_{L^{2}(\Omega)} \). 

We have
\begin{align*}
 \norm{ u }_{H^{1}(\Omega)} + \norm{ w }_{H^{1}_{0}(\Omega)}  &= \norm{u} + \norm{u_{x}} + \norm{w} + \norm{w_{x}} \\ 
 &= \left( \norm{u} + \norm{w} \right) + \left( \norm{u_{x}} + \norm{w_{x}} \right).
\end{align*}

From Lemmas \ref{lem:u-w-energy-estimate} and \ref{lem:ux-wx-energy-estimate}, the following bounds hold: 
\begin{itemize}
\item[(i)] For \( u^{m}, w^{m} \) as in the statement of the lemma:
\begin{equation*}
 \norm{u^{m}(t)} +  \norm{w^{m}(t)} 
 \leq 
 C_{1} \left( \norm{ u(0) } + \norm{ w(0) } \right) + C_{2}
\end{equation*}
for all \( t \in [0,T] \), where \( C_{1} , C_{2}  >0 \) are the positive constants, 
\begin{align*}
 C_{1} &= \sqrt{2} \max\left\{ e^{\frac{1}{2}(M^{2}+bM^{3}) T}, e^{\chi T} \right\}, \\
 C_{2} &= a e^{\chi T} \sqrt{2 \abs{ \Omega } T}  .
\end{align*}
\item[(ii)] 
For \( u_{x}, w_{x} \) as in the statement of the lemma,
\begin{align*}
 \|u_{x}^{m}(t)\| + \|w_{x}^{m}(t)\| &\leq
 D_{1} \left( \norm{u_{x}(0)} + \norm{w_{x}(0)} \right) + D_{2}
\end{align*}
for all \( t \in [0,T] \), where \( D_{1}, D_{2} > 0 \) are the positive constants, 
\begin{align*}
D_{1} &= \sqrt{2} e^{\tilde{D}_{1} T}, \\
D_{2} &= D_{1} \sqrt{2 \tilde{D_{2}} T} = 2 e^{\tilde{D}_{1} T} \sqrt{\tilde{D}_{2} T } ,
\end{align*}
and 
\begin{align*}
\tilde{D}_{1} &= 2 \max\left\{ 2M + 3bM^{3}, \frac{M^{2}+bM^{3}}{2} \right\} + 2 \max\left\{ M^{2}, 3bM^{3} \right\} + \max\left\{ \frac{d \norm{J}_{\text{op}}}{2}, \chi \right\} , \\
\tilde{D}_{2} &= \frac{d \norm{J}_{\text{op}}}{2} e^{2 (M^{2}+bM^{3}) T} \norm{ u(0) }^{2}.
\end{align*}
\end{itemize}

Hence, 
\begin{align*}
 \norm{ u }_{H^{1}(\Omega)} + \norm{ w }_{H^{1}_{0}(\Omega)}  &\leq \left( \norm{u} + \norm{w} \right) + \left( \norm{u_{x}} + \norm{w_{x}} \right) \\ 
 &\leq  C_{1} \left( \norm{ u(0) } + \norm{ w(0) } \right) + C_{2} \\ 
 &\qquad + D_{1} \left( \norm{u_{x}(0)} + \norm{w_{x}(0)} \right) +  D_{2} \\ 
 &=  C_{1} \left( \norm{ u(0) } + \norm{ w(0) } \right) + D_{2} \\ 
 &\qquad + D_{1} \left( \norm{u_{x}(0)} + \norm{w_{x}(0)} \right) +  C_{2} \\ 
\end{align*}
We have 
\begin{equation*}
 C_{1} \left( \norm{ u(0) } + \norm{ w(0) } \right) + D_{2} = 
 C_{1} \left( \norm{ u(0) } + \norm{ w(0) } \right) + D_{1} e^{(M^{2}+bM^{3}) T} \sqrt{ d \norm{J}_{\text{op}} } \norm{ u(0) }
\end{equation*}

Because of the factor \( \norm{ u(0) } \), let
\begin{equation*}
 E_{1} = 2 \max\left\{ C_{1}, D_{1} e^{(M^{2}+bM^{3}) T} \sqrt{d \norm{J}_{\text{op}} } \right\} .
\end{equation*}
Therefore, 
\begin{align*}
 C_{1} \left( \norm{ u(0) } + \norm{ w(0) } \right) + D_{2} \leq
 E_{1} \left( \norm{ u(0) } + \norm{ w(0) } \right) ,
\end{align*}
and the conclusion follows:
\begin{align*}
 \norm{ u }_{H^{1}(\Omega)} + \norm{ w }_{H^{1}_{0}(\Omega)}  &\leq \left( \norm{u} + \norm{w} \right) + \left( \norm{u_{x}} + \norm{w_{x}} \right) \\ 
 &\leq  E_{1} \left( \norm{ u(0) } + \norm{ w(0) } \right) \\ 
 &\quad + E_{2} \left( \norm{u_{x}(0)} + \norm{w_{x}(0)} \right) + E_{3}
\end{align*}
where
\begin{align*}
 E_{1} &= 2 \max\left\{ C_{1}, D_{1} e^{(M^{2}+bM^{3}) T} \sqrt{ d \norm{J}_{\text{op}} } \right\}, \\ 
 E_{2} &= D_{1} = \sqrt{2} e^{\tilde{D}_{1} T}, \\ 
 E_{3} &= C_{2} = a e^{\chi T} \sqrt{2 \abs{ \Omega } T} ,
\end{align*}
and the positive constants \( C_{1}, C_{2} \), \( D_{1} \), and \( \tilde{D}_{1} \) are defined in (\ref{eqn:C1}), (\ref{eqn:C2}), (\ref{eqn:D1}), and (\ref{eqn:tilde-D1}), respectively.
\end{proof}

\subsection{Uniform bounds}

\begin{lemma}
\label{lem:u-w-Linfty}
Let \( m \in \N\), and suppose that \( u^{m} \), \( w^{m} \) are functions of the form (\ref{eqn:galerkin-approxs}) which represent elements belonging to sequences of approximations to solutions of the first two equations in (\ref{eqn:model-aug-weak-proj2}). Then, there exist positive constants, \( \mathcal{C}_{1} \), \( \mathcal{C}_{2} \), and a time \( T \), such that, if
\begin{align*}
 \norm{ u(0) }_{L^{2}(\Omega)} + \norm{ w(0) }_{L^{2}(\Omega)} &\leq \mathcal{C}_{1}, \\ 
 \norm{ u_{x}(0) }_{L^{2}(\Omega)} + \norm{ w_{x}(0) }_{L^{2}(\Omega)} &\leq \mathcal{C}_{2} ,
\end{align*}
for all \( t \in [0,T] \), then
\begin{align}
\label{ineq:u-Linfty}
\norm{u^{m}}_{L^{\infty}(\Omega)} &< M/2 , \\  
\label{ineq:w-Linfty}
\norm{w^{m}}_{L^{\infty}(\Omega)} &< M/2 .
\end{align}
\end{lemma}
\begin{proof}
By the Sobolev embedding mentioned at the start of this section, there exists a constant \( C \), depending only on the problem dimension and $L^{p}$ space we work with (here, $n = 1$ and $p=2$), such that we have
\begin{equation}
\label{u-w-embedding-H1-Linfty}
 \norm{ u }_{L^{\infty}(\Omega)} + \norm{ w }_{L^{\infty}(\Omega)} \leq 
 C \left( \norm{ u }_{H^{1}(\Omega)} + \norm{ w }_{H^{1}_{0}(\Omega)} \right) .
\end{equation}

Thus, we choose initial values, \( u(0) \), \( w(0) \), \( u_{x}(0) \), \( w_{x}(0) \), such that the following linear combination of their \( L^{2}(\Omega) \)-norms satisfies:
\begin{equation}
\label{ineq:M5-1}
 C \left( E_{1} \left( \norm{ u(0) } + \norm{ w(0) } \right) + E_{2} \left( \norm{u_{x}(0)} + \norm{w_{x}(0)} \right) \right) \leq M / 5, 
\end{equation}
where \( E_{1} \) and \( E_{2} \) are positive constants defined in (\ref{eqn:E1}) and (\ref{eqn:E2}), respectively. 

Then, choose a positive time value \( T > 0 \) so small that 
\begin{equation}
\label{ineq:M5-2}
 C E_{3} = C a e^{\chi T} \sqrt{ 2 \abs{ \Omega } T } \leq M / 5 .
\end{equation}

Therefore, combining the estimates (\ref{ineq:M5-1}) and (\ref{ineq:M5-2}), we arrive at:
\begin{align*}
 \norm{ u }_{L^{\infty}(\Omega)} + \norm{ w }_{L^{\infty}(\Omega)} &\leq 
 C \left( \norm{ u }_{H^{1}(\Omega)} + \norm{ w }_{H^{1}_{0}(\Omega)} \right) \\
 &\leq C E_{1} \left( \norm{ u(0) } + \norm{ w(0) } \right) \\ 
 &\qquad + C E_{2} \left( \norm{u_{x}(0)} + \norm{w_{x}(0)} \right) \\ 
 &\qquad + C E_{3} \\ 
 &\leq 2 M /5 < M /2 .
\end{align*}
The conclusion follows, with 
\begin{align*}
 \mathcal{C}_{1} &= M / (2 C E_{1}), \\
 \mathcal{C}_{2} &= M / (2 C E_{2}) , 
\end{align*}
and \( T \) small enough that the above bounds hold.

\end{proof}

\begin{remark}
\label{concl:2}
Remark \ref{rem:u-w-H1} and Lemma \ref{lem:u-w-Linfty} imply that if we pick a time \( T \) and initial conditions small enough, looking for solutions to the equations, (\ref{eqn:model-aug-weak-proj2}), is equivalent to solving:
\begin{equation}
\label{eqn:model-weak-proj}
\begin{aligned}
 \left( u^{m}_{t}, \varphi_{k} \right)_{L^{2}(\Omega)} 
+ d B[u^{m},\varphi_{k}] + \mu \left( u^{m}, \varphi _{k} \right)_{L^{2}(\Omega)} &= \left( (u^{m})^{2} w^{m} ( 1-b u^{m} ) , \varphi_{k} \right)_{L^{2}(\Omega)} ,  \\ 
 \left( w^{m}_{t}, \psi_{k} \right)_{L^{2}(\Omega)} 
+ G[w^{m}, \psi_{k}] + \left( w^{m}, \psi_{k} \right)_{L^{2}(\Omega)} &= \left( a - (u^{m})^{2} w^{m}, \psi_{k} \right)_{L^{2}(\Omega)} .
\end{aligned}
\end{equation}

Namely, for such \( T >0 \), \( u^{m}, w^{m} \) are in \( L^{\infty}(\Omega) \) with upper bound \( M/2 \). 
The definition of \( \sigma \) (see (\ref{eqn:sigma})) then implies that, for all \( (x,t) \in \Omega \times [0,T] \), 
\begin{align*}
 \sigma(u^{m}(x,t)) &= u^{m}(x,t), \\ 
 \sigma(w^{m}(x,t)) &= w^{m}(x,t) .
\end{align*}

Thus, in the remainder of  this paper we assume that the constant \( M \) given  in the definition of the bounding function, \( \sigma \) (see equation (\ref{eqn:sigma})), the existence time  \( T \) obtained in Lemma \ref{lem:uniqueness}, and our initial conditions  are such that Lemma \ref{lem:u-w-Linfty} holds.

\end{remark}

\section{Convergence of Approximating Sequences}
\label{sec:5}

In this section, we build on the previous sections' results to conclude with the proof of the main theorem (Theorem \ref{thm:main}). First, in Lemma \ref{lem:u-w-ut-wt-bounded}, we show  that the sequences of approximate solutions, \( \left\{ u^{m} \right\} \), \( \left\{ w^{m} \right\} \), which represent elements belonging to sequences of approximations to solutions of the two equations in (\ref{eqn:model-weak-proj}), 
are bounded in 
\begin{equation*}
 L^{2}\left( 0,T; H^{1}(\Omega) \right) \times 
 L^{2}\left( 0,T; H^{1}_{0}(\Omega) \right) 
\end{equation*}
and \( \left\{ u^{m}_{t} \right\} \), \( \left\{ w^{m}_{t} \right\} \) are bounded in 
\begin{equation*}
L^{2}\left( 0,T; V^{\prime} \cap H^{-1}(\Omega) \right) \times 
 L^{2}\left( 0,T; H^{-1}(\Omega) \right) .
\end{equation*}

(Recall that \( V^{\prime} \) denotes the dual of \( V \), defined in (\ref{eqn:V}).) Next, in Subsection \ref{sub:5_1}, we establish weak convergence of the linear and nonlinear terms of our equations and conclude with a proof of the main theorem in Subsection \ref{sub:5_2}.

\begin{lemma}
\label{lem:u-w-ut-wt-bounded}
Let \( m \in \N\), and suppose that \( u^{m} \), \( w^{m} \) are functions of the form (\ref{eqn:galerkin-approxs}) which represent elements belonging to sequences of approximations to solutions of the two equations in (\ref{eqn:model-weak-proj}). Then, there exist positive constants \( F_{i} > 0 \), \( i = 1, 2, 3 \), satisfying
\begin{align}
 \label{eqn:u-energy-estimate-B}
 \norm{ u^{m} }_{L^{2}\left( 0,T; H^{1}(\Omega) \right)} &\leq F_{1} , \\ 
 \label{eqn:w-energy-estimate-B}
 \norm{ w^{m} }_{L^{2}\left( 0,T; H^{1}_{0}(\Omega) \right)} &\leq F_{1} ,
\end{align}
and 
\begin{align}
 \label{eqn:ut-energy-estimate-Vprime}
 \norm{ u^{m}_{t} }_{L^{2}\left( 0,T; V^{\prime} \right)} &\leq F_{2}  , \\
 \label{eqn:ut-energy-estimate}
 \norm{ u^{m}_{t} }_{L^{2}\left( 0,T; H^{-1}(\Omega) \right)} &\leq F_{2} , \\
 \label{eqn:wt-energy-estimate}
 \norm{ w^{m}_{t} }_{L^{2}\left( 0,T; H^{-1}(\Omega) \right)} &\leq F_{3}
\end{align}
where 
\begin{align}
 \label{eqn:F1} 
 F_{1} &= \sqrt{T} \left( E_{1} \left[ \norm{u(0)} + \norm{w(0)} \right) + E_{2} \left( \norm{u_{x}(0)} + \norm{w_{x}(0)} \right) \right] , \\
 \label{eqn:F2}
 F_{2} &=  \sqrt{T} e^{(M^{2}+bM^{3}) T}  \left(  \frac{M^{2}}{4} + b \frac{M^{3}}{8} + \mu + C d \right) \norm{ u(0) } ,  \\ 
 \label{eqn:F3}
 F_{3} &= \sqrt{T} \left( \frac{a M^{2} \sqrt{\abs{\Omega}}}{4} + C + 1 \right) F_{1}.
\end{align}

Here, \( E_{1} \) and \( E_{2} \) are the constants defined in Lemma \ref{lem:u-w-energy-estimate-H1}, in equations (\ref{eqn:E1}) and (\ref{eqn:E2}), respectively.
\end{lemma}

\begin{proof}
We drop the superscript \( m \), and write \( u = u^{m} \), \( w = w^{m} \), \( u_{x} = u_{x}^{m} \), and \( w_{x} = w_{x}^{m} \). We also write \( \norm{ \cdot } = \norm{ \cdot }_{L^{2}(\Omega)} \) and \( \left( \cdot, \cdot \right) = \left( \cdot, \cdot \right)_{L^{2}(\Omega)} \).

\textbf{Estimates for \( u^{m} \) and \( w^{m} \).}
We have \( \norm{ u }_{L^{2} (0,T; H^{1}(\Omega)) }^{2} = \int_{0}^{T} \norm{ u(t) }_{ H^{1}(\Omega) }^{2} \, dt \). By the estimate (\ref{eqn:u-w-H1}) for \( \norm{u}_{H^{1}(\Omega)} + \norm{w}_{H^{1}_{0}(\Omega)} \) proven in Lemma \ref{lem:u-w-energy-estimate-H1}, we have 
\begin{align*}
\left( \int_{0}^{T} \norm{ u(t) }_{H^{1}(\Omega)}^{2} \, dt \right)^{1/2} &\leq 
 \left[ T \left[ E_{1} \left( \norm{u(0)} + \norm{w(0)} \right) + E_{2} \left( \norm{u_{x}(0)} + \norm{w_{x}(0)} \right) \right]^{2} \right]^{1/2} \\
 \norm{ u }_{L^{2} (0,T; H^{1}(\Omega)) } &\leq \sqrt{T} \left[ E_{1} \left( \norm{u(0)} + \norm{w(0)} \right) + E_{2} \left( \norm{u_{x}(0)} + \norm{w_{x}(0)} \right) \right] ,
\end{align*}
and in the same way, 
\begin{equation*}
 \norm{ w }_{L^{2} (0,T; H^{1}_{0}(\Omega)) } \leq \sqrt{T} \left[ E_{1} \left( \norm{u(0)} + \norm{w(0)} \right) + E_{2} \left( \norm{u_{x}(0)} + \norm{w_{x}(0)} \right) \right] .
\end{equation*}
Letting \( F_{1} = \sqrt{T} \left[ E_{1} \left( \norm{u(0)} + \norm{w(0)} \right) + E_{2} \left( \norm{u_{x}(0)} + \norm{w_{x}(0)} \right) \right]  \), we obtain the estimates (\ref{eqn:u-energy-estimate-B}) and (\ref{eqn:w-energy-estimate-B}).

\textbf{Estimates for \( u^{m}_{t} \).}
Next, \( \norm{ u_{t} }_{ L^{2} (0,T; V^{\prime}) }^{2} = \int_{0}^{T} \norm{ u_{t} }_{V^{\prime}}^{2} \, dt \). 
The square root of the integrand is the dual space norm: \( \norm{ u_{t} }_{V^{\prime}} = \sup\left\{ \abs{ \left\langle u_{t}, v \right\rangle_{\left( V^{\prime} \times V \right)} } \mid v \in V, \; \norm{ v }_{V} \leq 1 \right\} \). Recall that \( V \sim L^{2}(\Omega) \). Therefore, we fix an element \( v \in V \) with \( \norm{v} \leq 1 \), and decompose \( v= v_{1} + v_{2} \), where \( v^{1}\in \vspan\left\{ \varphi_{k} \right\}_{k=0}^{2m} \) and \( v^{2} \) satisfies the orthogonality relation: \( \left( v^{2}, \varphi_{k} \right) \) for all \( k=0,1,\cdots, 2m \).

We remind the reader that \( \varphi_{k} \) are the basis functions defined in (\ref{eqn:basis-plant}) for the space \( V \) which are either \( \cos \) or \( \sin \) functions. 

Then, we identify \( u_{t} = \sum_{k=0}^{2m} d^{\prime}_{k}(t) \varphi_{k} \) with an element in \( V^{\prime} \sim L^{2}(\Omega) \). Since we are working in a Hilbert space, the dual pairing coincides with the inner product, and \( \left\langle u_{t}, v \right\rangle_{\left( V^{\prime} \times V \right)} = 
 \left( u_{t}, v^{1} \right) \), where we used our assumption that \( v = v^{1} + v^{2} \) where \( v^{2} \perp \vspan\left\{ \varphi_{k} \right\}_{k=0}^{2m} \). Then, \( \norm{ v^{1} } \leq \norm{ v } \leq 1 \). Using the first equation in  (\ref{eqn:model-weak-proj}), we have 
\begin{equation}
 \label{eqn:ut-v1-01}
 \left\langle u_{t}, v \right\rangle_{\left( V^{\prime} \times V \right)} = 
 \left( u_{t}, v^{1} \right) = 
 \left( u^{2} w (1 - bu) , v^{1} \right) - \mu \left( u, v^{1} \right) - d B[u, v^{1}] .
\end{equation}

Then, 
\begin{equation}
\label{eqn:ut-v1-estimate}
 \abs{ \left\langle u_{t}, v^{1} \right\rangle_{\left( V^{\prime} \times V \right)} } \leq 
 \abs{ \left( u^{2} w(1-bu), v^{1} \right) } + \mu \abs{ \left( u, v^{1} \right) }  + d B[u, v^{1}].
\end{equation}
We have, by Cauchy-Schwarz: \( \abs{ \left( u^{2} w(1-bu), v^{1} \right) } \leq \norm{ u^{2} w(1-bu) } \norm{v^{1}}  \) and \( \abs{ \left( u, v^{1} \right) } \leq \norm{ u } \norm{ v^{1} }\). By boundedness of the bilinear form \( B[\cdot,\cdot] \) (see Lemma \ref{lem:bilinear-form-plant}), there exists a constant \( C \) such that 
\begin{equation*}
 B[u, v^{1}] \leq 
 C \norm{u} \norm{v^{1}} .
\end{equation*}

Since \( \norm{v^{1}} \leq \norm{v} \leq 1 \), we conclude that:
\begin{equation*}
\abs{ \left( u^{2} w(1-bu), v^{1} \right) } \leq \norm{ u^{2} w(1-bu) }, \qquad 
\abs{ \left( u, v^{1} \right) } \leq \norm{ u }, \quad \text{ and } \quad 
B[u, v^{1}] \leq  C \norm{u} .
\end{equation*}

Then, by the triangle inequality, \( \norm{ u^{2} w(1-bu) } \leq \norm{ u^{2} w } + b \norm{ u^{3} w }  \), and using the \( L^{\infty}(\Omega) \)-estimates (\ref{ineq:u-Linfty}) and (\ref{ineq:w-Linfty})  for \( u \) and \( w \), respectively, we have 
\begin{align*}
 \norm{ u^{2} w(1-bu) } &\leq \left( \int_{\Omega} ( u^{2} w )^{2} \, dx \right)^{1/2} + b \left( \int_{\Omega} ( u^{3} w )^{2} \, dx \right)^{1/2} \\ 
 &\leq \left( \left( \frac{M}{2} \right)^{4} \int_{\Omega} u^{2} \, dx \right)^{1/2} + b \left( \left( \frac{M}{2} \right)^{6} \int_{\Omega} u^{2} \, dx \right)^{1/2} \\ 
 &= \left( \frac{M^{2}}{4} + b \frac{M^{3}}{8} \right) \norm{ u } .
\end{align*}

Hence, (\ref{eqn:ut-v1-estimate}) implies
\begin{align*}
 \abs{ \left\langle u_{t}, v^{1} \right\rangle_{\left( V^{\prime} \times V \right)} } &\leq \left( \frac{M^{2}}{4} + b \frac{M^{3}}{8} + \mu + C d\right) \norm{ u }
\end{align*}
Since \( v^{1} \in V \) with \( \norm{v^{1}} \leq 1 \), when we consider the supremum of the above quantity over all such \( v^{1} \), we obtain
\begin{equation*}
 \norm{ u_{t} }_{V^{\prime}} \leq \left( \frac{M^{2}}{4} + b \frac{M^{3}}{8} + \mu + C d\right) \norm{ u } .
\end{equation*}
Thus, 
\begin{align*}
 \norm{ u_{t}}_{L^{2}(0,T; V^{\prime})} &= \left( \int_{0}^{T} \norm{ u_{t} }_{V^{\prime}}^{2} \, dx  \right)^{1/2} \leq \left[ T \left( \frac{M^{2}}{4} + b \frac{M^{3}}{8} + \mu + C d \right)^{2} \norm{ u }^{2} \right]^{1/2} \\
 &\leq \sqrt{T} e^{(M^{2}+bM^{3}) T}  \left(  \frac{M^{2}}{4} + b \frac{M^{3}}{8} + \mu + C d \right) \norm{ u(0) } ,
\end{align*}
where in the third line, we applied the estimate (\ref{eqn:u-energy-estimate2}); that is, 
\begin{equation*}
 \max_{0 \leq t \leq T} \norm{ u(t) }^{2} \leq e^{2 (M^{2} + bM^{3}) T} \norm{ u(0) }^{2} .
\end{equation*}

This gives us (\ref{eqn:ut-energy-estimate-Vprime}) with \( F_{2} = \sqrt{T} e^{(M^{2}+bM^{3}) T}  \left(  \frac{M^{2}}{4} + b \frac{M^{3}}{8} + \mu + C d \right) \norm{ u(0) } \), as in (\ref{eqn:F2}). 

Now we estimate \( u_{t} \) in the \( L^{2}\left( 0,T; H^{-1}(\Omega) \right) \)-norm. 
Notice that \( \vspan \left\{ \varphi_{k} \right\} \subset H^{1}(\Omega) \). That is, because the basis functions \( \varphi_{k} \), \( k=0,1,\cdots, 2m \) are either \( \cos \) or \( \sin \) functions, one has that a constant multiple of \( \partial_{x} \varphi_{k} \) lies in \( V \sim L^{2}(\Omega) \). 

Therefore, we fix an element \( v \in H^{1}(\Omega) \) with \( \norm{v}_{H^{1}(\Omega)} \leq 1 \). We decompose \( v = v^{1} + v^{2} \), where \( v^{1} \in \vspan \left\{ \varphi_{k} \right\}_{k=0}^{2m} \) and \( v^{2} \) satisfies \( v^{2} \perp \vspan \left\{ \varphi_{k} \right\}_{k=0}^{2m} \) for all \( k = 0, 1, \cdots, 2m \). 
Therefore, \( \norm{ v^{1} }_{H^{1}(\Omega)} \leq \norm{ v }_{H^{1}(\Omega)} \leq 1 \). Note that \( V \sim L^{2}(\Omega) \subset H^{-1}(\Omega) \), and we identify \( u_{t} = \sum_{k=0}^{2m} d^{\prime}_{k}(t) \varphi_{k} \) with an element in \( H^{-1}(\Omega) \). Hence, similar to before, we have: 
\begin{equation}
 \label{eqn:ut-v1-02}
 \left\langle u_{t}, v \right\rangle_{\left( H^{-1}(\Omega) \times H^{1}(\Omega) \right)} = 
 \left( u_{t}, v^{1} \right) = 
 \left( u^{2} w (1 - bu) , v^{1} \right) - \mu \left( u, v^{1} \right) - d B[u, v^{1}] .
\end{equation}
Namely, the dual pairing \( \left\langle u_{t}, v \right\rangle_{\left( H^{-1}(\Omega) \times H^{1}(\Omega) \right)} = 
 \left( u_{t}, v^{1} \right) \) and we used the decomposition \( v = v^{1} + v^{2} \) where \( v^{2} \perp \vspan\left\{ \varphi_{k} \right\}_{k=0}^{2m} \). Since the right-hand side of (\ref{eqn:ut-v1-02}) is the same as the right-hand side for (\ref{eqn:ut-v1-01}), we conclude that after realizing the same process as above, the estimate (\ref{eqn:ut-energy-estimate}) holds.

\textbf{Estimate for \( w^{m}_{t} \).}
We finally estimate \( w_{t} \) in the \( L^{2}\left( 0,T; H^{-1}(\Omega) \right) \)-norm. 
Fix an element \( z \in H^{1}_{0}(\Omega) \), \( \norm{z}_{H^{1}_{0}(\Omega)} \leq 1 \) with \( z = z^{1} + z^{2} \), where \( z^{1} \in \vspan \left\{ \psi_{k} \right\}_{k=1}^{m} \) and \( z^{2} \perp \vspan\left\{ \psi_{k} \right\}_{k=1}^{m} \). 
It follows that \( \norm{ z^{1} }_{H^{1}_{0}(\Omega)} \leq \norm{ z }_{H^{1}_{0}(\Omega)} \leq 1 \). 

Because \( w^{m}_{t} \in L^{2}(\Omega) \subset H^{-1}(\Omega)  \), using the second equation in (\ref{eqn:model-weak-proj}) gives us:
\begin{equation*}
 \left\langle w_{t}, z \right\rangle_{\left( H^{-1}(\Omega) \times H^{1}_{0}(\Omega) \right)} = 
 \left( w_{t}, z^{1} \right) = 
 \left( a - u^{2} w, z^{1} \right) - \left( w, z^{1} \right) - G[w, z^{1}]  ,
\end{equation*}
which leads to:
\begin{equation*}
 \abs{ \left\langle w_{t}, z^{1} \right\rangle_{\left( H^{-1}(\Omega) \times H^{1}_{0}(\Omega) \right)} } 
 \leq 
 \abs{ \left( a - u^{2} w, z^{1} \right) } + \abs{ \left( w, z^{1} \right) } + G[w, z^{1}] . 
\end{equation*}
We again employ Cauchy-Schwarz and boundedness of \( G[\cdot,\cdot] \) (see Lemma \ref{lem:bilinear-form-water}) to obtain: \( \abs{ \left( a - u^{2} w, z^{1} \right) } \leq \norm{ a - u^{2} w } \norm{z^{1}}  \), \( \abs{ \left( w, z^{1} \right) } \leq \norm{ w } \norm{ z^{1} }\), and the existence of a constant \( C \) such that 
\begin{equation*}
 G[w, z^{1}] \leq 
 C \norm{w}_{H^{1}_{0}(\Omega)} \norm{z^{1}}_{H^{1}_{0}(\Omega)} .
\end{equation*}

Since \( \norm{z^{1}} \leq \norm{z^{1}}_{H^{1}_{0}(\Omega)} \leq 1 \), we use the triangle inequality to conclude that:
\begin{equation*}
\abs{ \left( a - u^{2} w, z^{1} \right) } \leq \norm{ a - u^{2} w } \leq \norm{ a } \norm{ u^{2} w }, \qquad 
\abs{ \left( w, z^{1} \right) } \leq \norm{ w }, \quad \text{ and } \quad 
G[w, z^{1}] \leq  C \norm{w}_{H^{1}_{0}(\Omega)} .
\end{equation*}

Note that \( \norm{a} = a \sqrt{ \abs{ \Omega } }  \).

We apply the \( L^{\infty}(\Omega) \)-estimate for \( u \) (see (\ref{ineq:u-Linfty}) to get that
\begin{equation*}
 \norm{ u^{2} w } 
 = \left( \int_{\Omega} ( u^{2} w )^{2} \, dx \right)^{1/2} 
 \leq \left( \left( \frac{M}{2} \right)^{4} \int_{\Omega} w^{2} \, dx \right)^{1/2} 
 = \frac{M^{2}}{4} \norm{ w } .
\end{equation*}

Because \( z^{1} \in H^{1}_{0}(\Omega) \) with \( \norm{ z^{1} }_{H^{1}_{0}(\Omega)} \leq 1 \) and \( \norm{ w } \leq \norm{ w }_{H^{1}_{0}(\Omega)} \), we have 
\begin{equation*}
 \norm{ w_{t} }_{H^{-1}(\Omega)} 
 \leq 
 \left( \frac{a M^{2} \sqrt{\abs{\Omega}}}{4} + C + 1  \right) \norm{w}_{H^{1}_{0}(\Omega)} 
\end{equation*}
Squaring and then integrating in from \( 0 \) to \( T \), one has 
\begin{align*}
 \norm{ w_{t} }_{L^{2}\left( 0,T; H^{-1}(\Omega) \right)}^{2}  = \int_{0}^{T} \norm{ w_{t} }_{H^{-1}(\Omega)}^{2} \, dt &\leq  \int_{0}^{T} \left( \frac{a M^{2} \sqrt{\abs{\Omega}}}{4} + C + 1  \right)^{2} \norm{w}_{H^{1}_{0}(\Omega)}^{2} \, dt \\ 
&= T \left( \frac{a M^{2} \sqrt{\abs{\Omega}}}{4} + C + 1  \right)^{2} \norm{ w }_{L^{2}(0,T, H^{1}_{0}(\Omega))}^{2}
\end{align*}

Now, we can apply the estimate (\ref{eqn:w-energy-estimate-B}) proved earlier; i.e., 
\begin{equation*}
\norm{ w }_{L^{2}(0,T, H^{1}_{0}(\Omega))} \leq F_{1}, 
\end{equation*}
where \( F_{1} \) a constant defined in (\ref{eqn:F1}). 
Thus, 
\begin{equation*}
\norm{ w_{t} }_{L^{2}\left( 0,T; H^{-1}(\Omega) \right)} \leq \sqrt{T} \left( \frac{a M^{2} \sqrt{\abs{\Omega}}}{4} + C + 1 \right) F_{1} =: F_{3}
\end{equation*}

\end{proof}

\subsection{Weak convergence}
\label{sub:5_1}

From Lemma \ref{lem:u-w-ut-wt-bounded} and the Banach-Alaoglu theorem we obtain that the sequences of approximate solutions and their time derivatives weakly converge to limits in $L^{2}$. This is the content of Lemma \ref{lem:weak-convergence}.
In the two lemmas that follow, we apply Lemma \ref{lem:weak-convergence} to obtain that, after integrating over the interval from $0$ to $T$, each of the terms appearing in the system equations (\ref{eqn:model-weak-proj}) converge to the appropriate limits. 

\begin{lemma}
\label{lem:weak-convergence}
Let \( m \in \N\), and suppose that \( u^{m} \), \( w^{m} \) are functions of the form (\ref{eqn:galerkin-approxs}) which represent elements belonging to sequences of approximations to solutions of the two equations in (\ref{eqn:model-weak-proj}). Then, there exist subsequences \( \left\{ u^{m_{l}} \right\}_{l=1}^{\infty} \subset \left\{ u^{m} \right\}_{m=1}^{\infty} \) and \( \left\{ w^{m_{l}} \right\}_{l=1}^{\infty} \subset \left\{ w^{m} \right\}_{m=1}^{\infty} \), and elements \( \left( \overline{u}, \overline{w} \right) \in L^{2}\left( 0,T;H^{1}(\Omega) \right) \times L^{2}\left( 0,T;H^{1}_{0}(\Omega) \right) \), with \( \left( \overline{u}_{t}, \overline{w}_{t} \right) \in L^{2}\left( 0,T;V^{\prime} \cap H^{-1}(\Omega) \right) \times L^{2}\left( 0,T;H^{-1}(\Omega) \right) \), such that 
\begin{equation}
\label{eqn:u-Banach-Alaoglu}
\begin{aligned}
 u^{m_{l}} &\rightharpoonup \overline{u} \quad \text{ weakly in } L^{2}\left( 0,T;H^{1}(\Omega) \right), \\
 u^{m_{l}}_{t} &\rightharpoonup \overline{u}_{t} \quad \text{ weakly in } L^{2}\left( 0,T;V^{\prime} \cap H^{-1}(\Omega) \right), \\
\end{aligned}
\end{equation}
and 
\begin{equation}
\label{eqn:w-Banach-Alaoglu}
 \begin{aligned}
 w^{m_{l}} &\rightharpoonup \overline{w} \quad  \text{ weakly in } L^{2}\left( 0,T;H^{1}_{0}(\Omega) \right), \\ 
 w^{m_{l}}_{t} &\rightharpoonup \overline{w}_{t} \quad \text{ weakly in } L^{2}\left( 0,T;H^{-1}(\Omega) \right) .
 \end{aligned}
\end{equation}
\end{lemma}

\begin{lemma}
\label{lem:weak-convergence-linear-terms}
Let \( m \in \N\), and suppose that \( u^{m} \), \( w^{m} \) are functions as in Lemma \ref{lem:weak-convergence}.
Then, after passing to a subsequence (which we do not relabel), the expressions
\begin{align*}
 & \int_{0}^{T} \left( u^{m}_{t}, \varphi_{k} \right)_{L^{2}(\Omega)} + d B[u^{m}, \varphi_{k}] + \mu \left( u^{m}, \varphi _{k} \right)_{L^{2}(\Omega)} \, dt \to \int_{0}^{T} \left( \overline{u}_{t}, \varphi_{k} \right)_{L^{2}(\Omega)} + d B[ \overline{u}, \varphi_{k} ] + \mu \left( \overline{u}, \varphi _{k} \right)_{L^{2}(\Omega)} \, dt,  \\ 
 & \int_{0}^{T} \left( w^{m}_{t}, \psi_{k} \right)_{L^{2}(\Omega)} + G[w^{m}, \psi_{k}] + \left( w^{m}, \psi_{k} \right)_{L^{2}(\Omega)} \, dt \to \int_{0}^{T} \left( \overline{w}_{t}, \psi_{k} \right)_{L^{2}(\Omega)} + G[\overline{w}, \psi_{k}] + \left( \overline{w}, \psi_{k} \right)_{L^{2}(\Omega)} \, dt,
\end{align*}
as $m\to \infty$.
\end{lemma}

\begin{proof}
All of the linear terms, except for \( \int_{0}^{T} B[u^{m},\varphi_{k}] \, dt \), can easily be seen to satisfy this convergence by Lemma \ref{lem:weak-convergence}.
To see this for \( \int_{0}^{T} B[u^{m},\varphi_{k}] \, dt \), recall that we consider nonlocal Dirichlet boundary conditions for the plant equation: \( u = 0 \) on \( \Omega^{c} \). Thus, for \( v \in C^{1}\left( [0,T]; V \right) \): 
\begin{align*}
 B[u^{m}, v] &= \frac{1}{2} \int_{\R} \int_{\R} \left( u^{m}(y,t) - u^{m}(x,t) \right) \gamma(x,y) \left( v(y,t) - v(x,t) \right) \, dy \, dx \\ 
 &= \overbrace{\frac{1}{2} \int_{\R} \int_{\R} u^{m}(y,t) \gamma(x,y) \left( v(y,t) - v(x,t) \right) \, dy \, dx}^{\text{(I)}} \\ 
 &\quad - \underbrace{\frac{1}{2} \int_{\R} u^{m}(x,t) \int_{\R} \gamma(x,y) \left( v(y,t) - v(x,t) \right) \, dy \, dx}_{\text{(II)}} .
\end{align*}
Swapping the order of integration in integral (I) and then symbollically swapping the variables \( x \) and \( y \), we obtain:
\begin{align*}
 \text{(I)} &= \frac{1}{2} \int_{\Omega} u^{m}(y) \int_{\R} \gamma(x,y) \left( v(y,t) - v(x,t) \right) \, dx \, dy \\ 
 &= \frac{1}{2} \int_{\Omega} u^{m}(x) \int_{\R} \gamma(x,y) \left( v(x,t) - v(y,t) \right) \, dy \, dx , \\ 
 &= - \frac{1}{2} \int_{\Omega} u^{m}(x) \int_{\R} \gamma(x,y) \left( v(y,t) - v(x,t) \right) \, dy \, dx , \\ 
\end{align*}
where in the second line we used symmetry of the kernel, \( \gamma \) (see Hypothesis \ref{hyp:kernel}). In the second integral, (II), we have 
\begin{equation*}
 \text{(II)} = \frac{1}{2} \int_{\Omega} u^{m}(x) \int_{\R} \gamma(x,y) \left( v(x,t) - v(y,t) \right) \, dy \, dx 
\end{equation*}

Then, 
\begin{align*}
 B[u^{m}, v] &= - \int_{\Omega} u^{m}(x,t) \int_{\R} \gamma(x,y) \left( v(y,t) - v(x,t) \right) \, dy \, dx \\ 
 &= - \int_{\Omega} u^{m}(x,t) g(x,t) \, dx, 
\end{align*}
where in the second line we defined \( g(x,t) = \gamma * v - \Gamma v \), with \( \Gamma = \int_{\R} \gamma(x,y) \, dy \).
If we can verify that \( g \in L^{2}\left( 0,T; V \right) \), then, after passing to a subsequence, we may conclude that \( \int_{0}^{T} B[u^{m},v] \, dt \to \int_{0}^{T} B[\overline{u}, v] \, dt \).

Indeed, by the triangle inequality,
\begin{align*}
 \norm{ g }_{L^{2}\left( 0,T; L^{2}(\Omega) \right)} &= \norm{ \gamma * v - \Gamma v }_{L^{2}\left( 0,T; L^{2}(\Omega) \right)} \\ 
 &\leq \norm{ \gamma * v }_{L^{2}\left( 0,T; L^{2}(\Omega) \right)} + \Gamma \norm{ v }_{L^{2}\left( 0,T; L^{2}(\Omega) \right)}. 
\end{align*}
By Hypothesis (\ref{hyp:kernel}), \( \gamma \in L^{1}(\R) \), so that by Young's inequality: 
\begin{align*}
 \norm{ \gamma * v }_{L^{2}(\Omega)} &\leq \norm{ \gamma * v }_{L^{2}(\R)} \\ 
 &\leq \norm{ \gamma }_{L^{1}(\R)} \norm{ v }_{L^{2}(\R)} = \norm{ \gamma }_{L^{1}(\R)} \norm{ v }_{L^{2}(\Omega)} .
\end{align*}

Recall that \( V \sim L^{2}(\Omega) \). Then, since \( \norm{ \gamma }_{L^{1}(\R)} \) is constant in time and \( v \in C^{1}\left( [0,T] ; V \right) \subset L^{2}\left( 0,T; V \right) \), it follows that \( g \in L^{2}\left( 0,T; V \right) \).
\end{proof}

\begin{remark}[The nonlinear case]
\label{rem:usingAubin}
Altogether, Lemma \ref{lem:weak-convergence}, the embeddings
\begin{equation*}
 H^{1}(\Omega) \ssubset L^{2}(\Omega) \subset H^{-1}(\Omega), \quad 
 H^{1}_{0}(\Omega) \ssubset L^{2}(\Omega) \subset H^{-1}(\Omega), 
\end{equation*}
and Aubin's Compactness Theorem,  imply that the sequence of approximations $\{ u^m\}$, $\{ w^m\}$ converge strongly to functions $\bar{u}$ and $\bar{w}$ in $L^2(0,T, L^2(\Omega))$. 
\end{remark}

\begin{lemma}
 \label{lem:weak-convergence-nonlinear-terms} 
Let \( m \in \N\), and suppose that \( u^{m} \), \( w^{m} \) are functions as in Lemmas \ref{lem:weak-convergence} and \ref{lem:weak-convergence-linear-terms}. Let \((v,z) \in C^{1}\left( [0,T], H^{1}(\Omega) \right) \times C^{1}\left( [0,T], H^{1}_{0}(\Omega) \right) \).
 Then, after passing to a subsequence (which we do not relabel), 
\begin{equation}
\label{eqn:nonlinear-weak-convergence} 
\begin{aligned}
 \int_{0}^{T} \int_{\Omega} (u^{m})^{2} w^{m} (1 - b u^{m}) \, {v} \, dx \, dt & \to 
 \int_{0}^{T} \int_{\Omega} \overline{u}^{2} \overline{w} (1 - b \overline{u}) \, {v} \, dx \, dt , \\ 
 \int_{0}^{T} \int_{\Omega}  (a - (u^{m})^{2} w^{m}) \, {z} \, dx \, dt & \to 
 \int_{0}^{T} \int_{\Omega} (a - \overline{u}^{2} \overline{w}) \, {z} \, dx \, dt , \\ 
\end{aligned}
\end{equation}
as \( m \to \infty \).
\end{lemma}

\begin{proof}

We demonstrate that after passing to a subsequence (which we again do not relabel),
\begin{align*}
 \int_{0}^{T} \int_{\Omega} \abs{ (u^{m})^{2} w^{m} (1 - b u^{m}) - \overline{u}^{2} \overline{w} (1 - b \overline{u}) } \abs{ v } \, dx \, dt \to & 0 , \\ 
\int_{0}^{T} \int_{\Omega} \abs{ (a - (u^{m})^{2} w^{m}) - (a - \overline{u}^{2} \overline{w}) } \abs{ z } \, dx \, dt \to & 0 ,
\end{align*}
as \( m\to \infty \). 
The primary tool is the Dominated Convergence theorem. Therefore, let 
\[ f^{1}_{m} = \abs{ (u^{m})^{2} w^{m} (1 - b u^{m}) - \overline{u}^{2} \overline{w} (1 - b \overline{u}) } \abs{ v }  \]
and
\[f^{2}_{m} = \abs{ (a - (u^{m})^{2} w^{m}) - (a - \overline{u}^{2} \overline{w}) } \abs{ z } .\] 
We achieve this in two steps: First, we show that \( f^{i}_{m} \to \overline{f} \equiv 0 \) for a.e. \( (x,t) \in \Omega \times [0,T] \), \( i=1,2 \), and then that there are non-negative \( g_{1}, g_{2} \in L^{1}(\Omega \times [0,T]) \) such that \( \abs{f^{i}_{m}} \leq g_{i} \) for all \( m \) and for a.e. \( (x,t) \in \Omega \times [0,T] \), \( i=1,2 \).

In Remark \ref{rem:usingAubin}  we noted that $\{ u^m\}, \{w^m\} $ converge strongly on \( L^{2}\left( 0,T,L^{2}(\Omega) \right)\). Next, notice that  \( L^{2}\left( 0,T,L^{2}(\Omega) \right)\subset L^{1}\left( \Omega \times [0,T] \right) \). This follows by Cauchy-Schwarz: For \( f \in L^{2}\left( 0,T, L^{2}(\Omega) \right) \),
\[
\int_{\Omega \times [0,T]} \abs{ f(x,t) } \, dx \, dt \leq 
\left[ \int_{0}^{T} \int_{\Omega} \abs{ f(x,t) }^{2} \, dx \, dt \right]^{1/2} \abs{ \Omega \times [0,T] }^{1/2} < \infty.
\]
As a result, the sequences \( \left\{ u^{m} \right\}_{m=1}^{\infty} \) and \( \left\{ w^{m} \right\}_{m=1}^{\infty} \) both converge in \( L^{1}\left( \Omega \times [0,T] \right) \). 
Then, by standard integration theory (see, e.g., \cite[Corollary 2.32]{folland1999real}), we deduce the existence of subsequences \( \left\{ u^{m} \right\} \) and \( \left\{ w^{m} \right\} \) (which we do not relabel) that converge pointwise a.e. to \( \overline{{u}} \) and \( \overline{w} \), respectively, on \( \Omega \times [0,T] \). 
Because \( F^{(1)}(u,w) = u^{2} w (1 - b u) \) and \( F^{(2)}(u,w) = a - u^{2} w \) are continuous functions of \( u, w \in \R \), it follows that that \( F^{(i)}(u^{m}, w^{m}) \) converges to \( F^{(i)}(\overline{u}, \overline{w}) \) a.e., for \( i=1,2 \).

Then because \( (v, z) \in C^{1}\left( [0,T], H^{1}(\Omega) \right) \times C^{1}\left( [0,T], H^{1}_{0}(\Omega) \right) \) by hypothesis, using the Sobolev embedding \cite[Theorem 4.12 in Chapter 4]{adams2003sobolev} (see also (\ref{u-w-embedding-H1-Linfty})), there exist constants \( C_{1} \), \( C_{2} \) satisfying 
\begin{align*}
 \norm{ v }_{L^{\infty}(\Omega)} &\leq C_{1} \norm{ v }_{H^{1}(\Omega)} , \\
 \norm{ z }_{L^{\infty}(\Omega)} &\leq C_{2} \norm{ z }_{H^{1}_{0}(\Omega)} .
\end{align*}

In other words, as we've seen before, such functions are essentially bounded on \( \Omega \times [0,T] \):
\begin{align*}
 \norm{ v }_{L^{\infty}(\Omega \times [0,T])} &\leq B_{1} , \\
 \norm{ z }_{L^{\infty}(\Omega \times [0,T])} &\leq B_{2} , 
\end{align*}
for some \( B_{1}, B_{2} > 0 \). Here and throughout, \( m \in \N\) is arbitrary.

Hence, we have 
\begin{align*}
 f^{1}_{m} = 
 \abs{ F^{(1)}(u^{m}, w^{m}) - F^{(1)}(\overline{u}, \overline{w}) } \abs{ v } &\leq 
 B_{1} \abs{ F^{(1)}(u^{m}, w^{m}) - F^{(1)}(\overline{u}, \overline{w})} \to 0 , \\
 f^{2}_{m} = \abs{ F^{(2)}(u^{m}, w^{m}) - F^{(2)}(\overline{u}, \overline{w}) } \abs{ z } &\leq 
 B_{2} \abs{ F^{(2)}(u^{m}, w^{m}) - F^{(2)}(\overline{u}, \overline{w}) } \to 0 .
\end{align*}

To prove item 2, here, we let \( C_{1} = \norm{v}_{L^{\infty}(\Omega \times [0,T])} \) and \( C_{2} = \norm{z}_{L^{\infty}(\Omega \times [0,T])} \). 

Lemma \ref{lem:u-w-Linfty} gives us that
\begin{align*}
 \norm{ u^{m} }_{L^{\infty}(\Omega \times [0,T])} \leq M/2, \\
 \norm{ w^{m} }_{L^{\infty}(\Omega \times [0,T])} \leq M/2 .
\end{align*}

Then, notice that since \( u^{m} \to \overline{u} \) a.e., we also have \( \norm{\overline{u}}_{L^{\infty}(\Omega \times [0,T])} \leq M/2 \). Similarly, \( \norm{ \overline{w} }_{L^{\infty}(\Omega \times [0,T])} \leq M/2 \).
Restricting the domains of \( F^{(i)} \) to the finite intervals \( [-M/2, M/2] \), \( i=1,2 \), the \( F^{(i)} \) attain their max. We put \( M_{i} =  \max\left\{  \abs{F^{(i)}(z)}: z \in [-M/2, M/2] \right\} \) for \( i=1,2 \). 

Let \( g_{i} = 2 C_{i} M_{i}  \) for \( i=1,2 \). We now show that \( \abs{f^{i}_{m}} \leq g_{i} \) for a.e. \( (x,t) \in \Omega \times [0,T] \), \( i=1,2 \). 
Indeed, by the triangle inequality:
\begin{equation*}
 f^{i}_{m} = \abs{ F^{(i)}(u^{m},w^{m}) - F^{(i)}(\overline{u},\overline{w}) } \abs{ v } \leq C_{i} \left( \abs{F^{(i)}(u^{m}, w^{m})} + \abs{F^{(i)}(\overline{u}, \overline{w})} \right) \leq 2  C_{i} M_{i} = g_{i}, \quad 
\end{equation*}
for \( i = 1,2  \).

Since \( g_{1} \), \( g_{2} \) are constants and the domain \( \Omega \times [0,T] \) is finite, it follows that \( g_{i} \in L^{1}\left( \Omega \times [0,T] \right) \), \( i=1,2 \). 

The conclusion follows.
\end{proof}

\subsection{Proof of the main theorem}
\label{sub:5_2}

\begin{theorem}[Main Theorem]
Let \( u_{0}, w_{0} \) be given functions in \( H^{1}(\Omega) \). Then, there exist positive constants \( \mathcal{C}_{1} \), \( \mathcal{C}_{2} \), and \( T \), such that if 
\begin{equation*}
 \norm{ u_{0} }_{L^{2}(\Omega)} + \norm{ w_{0} }_{L^{2}(\Omega)} < \mathcal{C}_{1}, \quad 
 \norm{ \partial_{x} u_{0} }_{L^{2}(\Omega)} + \norm{ \partial_{x} w_{0} }_{L^{2}(\Omega)} < \mathcal{C}_{2}, 
\end{equation*}
then the weak formulation of our model equations, 
\begin{align*}
 \left\langle u_{t}, \varphi \right\rangle_{\left( V^{\prime} \times V \right)} + d B[u, \varphi] + \mu \left( u, \varphi \right)_{L^{2}(\Omega)} &= 
 \left( u^{2} w (1 - b u), \varphi \right)_{L^{2}(\Omega)} \\ 
 \left\langle w_{t}, \psi \right\rangle_{\left( H^{-1}(\Omega) \times H^{1}_{0}(\Omega) \right)} + G[w, \psi] + \left( w, \psi \right)_{L^{2}(\Omega)} &= 
 \left( a - u^{2} w, \psi \right)_{L^{2}(\Omega)} ,
\end{align*}
has a solution \( (\overline{u}, \overline{w}) \in L^{2}\left( 0,T; H^{1}(\Omega) \right) \times L^{2}\left( 0,T; H^{1}_{0}(\Omega) \right) \) satisfying the initial conditions, \( \overline{u}(x,0) = u_{0} \) and \( \overline{w}(x,0) = w_{0} \).
\end{theorem}

\begin{proof}
Let \( u^{m} \), \( w^{m} \) denote functions of the form (\ref{eqn:galerkin-approxs}) that solve the first two equations in (\ref{eqn:model-weak-proj}). By Lemma \ref{lem:u-w-Linfty}, the constants \( \mathcal{C}_{1} \) and \( \mathcal{C}_{2} \) are given by 
\begin{align*}
 \mathcal{C}_{1} &= M / (2 C E_{1}), \\
 \mathcal{C}_{2} &= M / (2 C E_{2}) , 
\end{align*}
where \( E_{1} \), \( E_{2} \) are defined in (\ref{eqn:E1}), (\ref{eqn:E2}), respectively, and \( C \) is the positive constant appearing in the embedding (\ref{u-w-embedding-H1-Linfty}).

As in previous proofs, we write \( \left( \cdot, \cdot \right) = \left( \cdot, \cdot \right)_{L^{2}(\Omega)} \) and let \( F^{(1)}(u,w) = u^{2} w (1-b u) \) and \( F^{(2)}(u,w) = a - u^{2} w \).

We choose test functions 
\begin{equation*}
 (v, z) \in C^{1}\left( [0,T]; H^1(\Omega) \right) \times C^{1}\left( [0,T]; H^{1}_{0}(\Omega) \right)
\end{equation*}
of the form 
\begin{equation*}
 v(t) = \sum_{k=0}^{2 N} \alpha_{k}(t) \varphi_{k}  \quad \text{and} \quad 
 z(t) = \sum_{k=1}^{N} \beta_{k}(t) \psi_{k}, 
\end{equation*}
where $N \in \N$ is fixed and \( \alpha_{k}(t) \), \( \beta_{k}(t) \) are given smooth functions and \( \left\{ \varphi_{k} \right\}_{k=0}^{2m} \) and \( \left\{ \psi_{k} \right\}_{k=1}^{m} \) are finite subsets of bases for \( V \) and \( H^{1}_{0}(\Omega) \), respectively (see (\ref{eqn:basis-plant})). Then notice that 
\begin{align*}
\left( u^{m}_{t}, \varphi_{j} \right)_{L^{2}(\Omega)} 
+ d B[u^{m},\varphi_{j}] + \mu \left( u^{m}, \varphi _{j} \right)_{L^{2}(\Omega)} &= \left( (u^{m})^{2} w^{m} ( 1-b u^{m} ) , \varphi_{j} \right)_{L^{2}(\Omega)} ,  \\ 
\left( w^{m}_{t}, \psi_{k} \right)_{L^{2}(\Omega)} 
+ G[w^{m}, \psi_{k}] + \left( w^{m}, \psi_{k} \right)_{L^{2}(\Omega)} &= \left( a - (u^{m})^{2} w^{m}, \psi_{k} \right)_{L^{2}(\Omega)} ,
\end{align*}
hold for each $j = 0, 1, \cdots, 2 m$, $k = 1, 2, \cdots, m$.
So, choose \( m \geq N \) and multiply the first and second equations above by \( \alpha_{j}(t) \), \( j=0,1,\cdots, 2N \) and \( \beta_{k}(t) \), \( k = 1, 2, \cdots, N \), respectively, and sum. This yields: 
\begin{align*}
 \left( u^{m}_{t}, v \right) + d B[u^{m}, v] + \mu \left( u^{m}, v \right) &= 
 \left( F^{(1)}(u^{m},w^{m}), v \right) , \\ 
 \left( w^{m}_{t}, z \right) + G[w^{m}, z] + \left( w^{m}, z \right) &= 
 \left( F^{(2)}(u^{m},w^{m}), z \right) .
\end{align*}

We note that \( u^{m}_{t} \) and \( w^{m}_{t} \) are elements in 
\begin{equation*}
 L^{2}(\Omega) \sim \left( L^{2}(\Omega) \right)^{*} \sim V^{\prime} \quad 
 \text{and} \quad 
 L^{2}(\Omega) \subset H^{-1}(\Omega) ,  
\end{equation*}
respectively. Because the spaces above are Hilbert spaces, integrating the centered equations above with respect to \( t \) then gives: 
\begin{align*}
 \int_{0}^{T} \left\langle u^{m}_{t}, v \right\rangle_{\left( V^{\prime} \times V \right)} + d B[u^{m}, v] + \mu \left( u^{m}, v \right) \, dt &= \int_{0}^{T} \left( F^{(1)}(u^{m}, w^{m}), v \right) \, dt , \\ 
 \int_{0}^{T} \left\langle w^{m}_{t}, z \right\rangle_{\left( H^{-1}(\Omega) \times H^{1}_{0}(\Omega) \right)} + G[w^{m}, z] + \left( w^{m}, z \right) \, dt &= \int_{0}^{T} \left( F^{(2)}(u^{m}, w^{m}), z \right) \, dt .
\end{align*}

Passing to a subsequence (which we do not relabel) and subsequent weak limits, our choice of test functions \( v \) and \( z \) allow us to invoke Lemmas \ref{lem:weak-convergence-linear-terms} and \ref{lem:weak-convergence-nonlinear-terms} to find that 
\begin{equation}
\label{eqn:weak-limit}
\begin{aligned}
 \int_{0}^{T} \left\langle \overline{u}_{t}, v \right\rangle_{\left( V^{\prime} \times V \right)} + d B[\overline{u}^{m}, v] + \mu \left( \overline{u}, v \right) \, dt &= \int_{0}^{T} \left( F^{(1)}(\overline{u}, \overline{w}), v \right) \, dt , \\ 
 \int_{0}^{T} \left\langle \overline{w}_{t}, z \right\rangle_{\left( H^{-1}(\Omega) \times H^{1}_{0}(\Omega) \right)} + G[\overline{w}, z] + \left( \overline{w}, z \right) \, dt &= \int_{0}^{T} \left( F^{(2)}(\overline{u}, \overline{w}), z \right) \, dt .
\end{aligned}
\end{equation}

Since the space $V$ is equivalent to $L^2(\Omega)$ we have that the spaces \( C^{1}\left( [0,T] ; H^1(\Omega) \right) \) and \( C^{1}\left( [0,T]; H^{1}_{0}(\Omega) \right) \) are dense in \( L^{2}\left( 0,T; V \right) \) and \( L^{2}\left( 0,T; H^{1}_{0}(\Omega) \right) \), respectively. Therefore,  the above expressions hold for all \( (v,z) \in L^{2}\left( 0,T; V \right) \times L^{2}\left( 0,T; H^{1}_{0}(\Omega) \right) \).

Then, applying \cite[Theorem 2 in Section 5.9]{evans2022partial}, it follows that \( \overline{u}, \overline{w} \in C\left( [0,T]; L^{2}(\Omega) \right) \).

Next, let \( u_{0} = g_{1} \) and \( w_{0} = g_{2} \). To show that \( \overline{u}(x,0) = g_{1} \), \( \overline{w}(x,0_) = g_{2} \), we integrate by parts: 
\begin{align*}
 \int_{0}^{T} \left\langle \overline{u}_{t}, v \right\rangle_{\left( V^{\prime} \times V \right)} \, dt &= - \int_{0}^{T} \left\langle v_{t}, \overline{u} \right\rangle_{\left( V^{\prime} \times V \right)} \, dt + \left( \overline{u}(T), v(T) \right) - \left( \overline{u}(0), v(0) \right) , \\ 
 \int_{0}^{T} \left\langle \overline{w}_{t}, z \right\rangle_{\left( H^{-1}(\Omega) \times H^{1}_{0}(\Omega) \right)} \, dt &= - \int_{0}^{T} \left\langle z_{t}, \overline{w} \right\rangle_{\left( H^{-1}(\Omega) \times H^{1}_{0}(\Omega) \right)} \, dt + \left( \overline{w}(T), z(T) \right) - \left( \overline{w}(0), z(0) \right)
\end{align*}

Suppose \( v(T) = z(T) = 0 \). Then inserting the above into (\ref{eqn:weak-limit}), we obtain:
\begin{equation}
\label{eqn:weak-limit2}
\begin{aligned}
 - \int_{0}^{T} \left\langle v_{t}, \overline{u} \right\rangle_{\left( V^{\prime} \times V \right)} + d B[\overline{u}, v] + \mu \left( \overline{u}, v \right) \, dt &= \int_{0}^{T} \left( F^{(1)}(\overline{u}, \overline{w}), v \right) \, dt + \left( \overline{u}(0), v(0) \right) , \\
 - \int_{0}^{T} \left\langle z_{t}, \overline{w} \right\rangle_{\left( H^{-1}(\Omega) \times H^{1}_{0}(\Omega) \right)} +  G[\overline{w}, z] + \left( \overline{w}, z \right) \, dt &= \int_{0}^{T} \left( F^{(2)}(\overline{u}, \overline{w}), z \right) \, dt + \left( \overline{w}(0), z(0) \right)
\end{aligned}
\end{equation}

Similarly, the sequences of approximations satisfy 
\begin{align*}
 - \int_{0}^{T} \left\langle v_{t}, u^{m} \right\rangle_{\left( V^{\prime} \times V \right)} + d B[u^{m}, v] + \mu \left( u^{m}, v \right) \, dt &= \int_{0}^{T} \left( F^{(1)}(u^{m}, w^{m}), v \right) \, dt \\ 
 & \qquad + \left( u^{m}(0), v(0) \right) , \\
 - \int_{0}^{T} \left\langle z_{t}, w^{m} \right\rangle_{\left( H^{-1}(\Omega) \times H^{1}_{0}(\Omega) \right)} +  G[w^{m}, z] + \left( w^{m}, z \right) \, dt &= \int_{0}^{T} \left( F^{(2)}(u^{m}, w^{m}), z \right) \, dt \\ 
 & \qquad + \left( w^{m}(0), z(0) \right)
\end{align*}

Recall that \( u^{m} = \sum_{k=0}^{2m} d_{k}(t) \varphi_{k} \) and \( w^{m} = \sum_{k=1}^{m} e_{k}(t) \psi_{k} \). We choose the coefficients \( d_{k}, e_{k} \) to be smooth and satisfy 
\begin{align*}
 d_{k}(0) &= \left( g_{1}, \varphi_{k} \right), \quad k =0,1,\cdots, 2m, \\ 
 e_{k}(0) &= \left( g_{2}, \psi_{k} \right), \quad k =1,2,\cdots, m .
\end{align*}

Hence, \( u^{m}(0) \to g_{1} \) and \( w^{m}(0) \to g_{2} \) in \( L^{2}(\Omega) \). Then, setting \( m = m_{l} \) and passing to weak limits, we obtain
\begin{align*}
 - \int_{0}^{T} \left\langle v_{t}, \overline{u} \right\rangle_{\left( V^{\prime} \times V \right)} + d B[\overline{u}, v] + \mu \left( \overline{u}, v \right) \, dt &= \int_{0}^{T} \left( F^{(1)}(\overline{u}, \overline{w}), v \right) \, dt + \left( g_{1}, v(0) \right) , \\
 - \int_{0}^{T} \left\langle z_{t}, \overline{w} \right\rangle_{\left( H^{-1}(\Omega) \times H^{1}_{0}(\Omega) \right)} +  G[\overline{w}, z] + \left( \overline{w}, z \right) \, dt &= \int_{0}^{T} \left( F^{(2)}(\overline{u}, \overline{w}), z \right) \, dt + \left( g_{2}, z(0) \right) .
\end{align*}

Then because \( v(0), z(0) \in L^{2} (\Omega) \) are arbitrary, comparing the above with (\ref{eqn:weak-limit2}), we conclude that \( \overline{u}(x,0) = g_{1} \), \( \overline{w}(x,0) = g_{2} \).

\end{proof}

\section{Discussion}

In this paper, we adapt the Galerkin approach to prove the existence and uniqueness of weak solutions for a modified version of the Klausmeier model, a system of coupled equations that describes the interplay between water and plant biomass in arid ecosystems. In contrast to the original model, which uses a Laplace operator to describe the spread of plant seeds, our equations employ a convolution operator to account for long-range dispersal events caused by wind and animals. The resulting system of integro-differential equations is then posed on a 1-d bounded domain, $\Omega$, and supplemented by homogeneous Dirichlet boundary constraints modeling a desert state on $\R \setminus \Omega$. For the water equation, these correspond to local boundary conditions prescribed along the boundary, while for the plant equation they take the form of an additional constraint that requires the value of plant biomass to be zero on $\R \setminus \Omega$. 

In our analysis, the main difficulty comes from the nonlocal operator, which is defined by a positive, symmetric, and spatially extended convolution kernel with finite second moment. Because the corresponding bilinear form then yields energy estimates that only bound the sequence of Galerkin approximations for the plant variable in $L^2(\Omega)$ and not in $H^1(\Omega)$, we are not able to directly use compactness arguments to prove the weak convergence of the nonlinear terms. To achieve the required regularity we included equations for two additional unknown functions, which we then showed correspond to the spatial derivatives of the water and plant biomass variables. The result is a proof for the short-time existence and uniqueness of weak solutions to the nonlocal Klausmeier model under the assumption of small initial data in $H^1(\Omega)$.

Although we consider here only the 1-d case, we expect that our results will remain valid in higher dimensions. In fact, the same approach can be carried out without major modifications when considering bounded domains in $\R^n$. The only caveat is that in order to apply the required compactness result, one needs to bound the sequence of Galerkin approximations in a Sobolev space with higher regularity. This can be achieved again by including equations for additional variables representing derivatives of the unknown functions. Notice that the work needed to prove the energy bounds for these extra variables is straightforward but indeed more cumbersome. Consequently, we optimistically take the present work as a guarantee that any finite element discretization of the nonlocal Klausmier model in $\Omega \subset \R^n$ is well-posed. Our future work in this direction will follow \cite{cappanera2024analysis} in order to carry out numerical simulations of this set of integro-differential equations, now posed on a 2-d bounded domain. 

We emphasize that the approach presented here comes with some limitations. Mainly, it does not deliver a proof that, subject to nonnegative initial conditions, our solutions will remain nonnegative, which is expected if the equations are to model an arid ecosystem. Similarly, our method does not guarantee global existence of these solutions. These two properties can be obtained, for example, by assuming that the unknowns live in $L^\infty$ and using semigroup theory. Indeed, as mentioned in the introduction, previous work has successfully used this method to prove  global existence of nonnegative solutions for a nonlocal Gray-Scott model \cite{Philippe2023}. We expect that a similar approach can be used in the present case, and this is part of our future work. Notice, however, that it is not clear if this technique can be applied to our problem when considering nonlocal Neumann boundary conditions for the plant equation, since the nonlocal operator in this case may not give rise to a semigroup in $L^\infty(\Omega)$. In this sense, the method presented in this article provides, to our knowledge, the only viable approach to prove the existence of (weak) solutions for nonlocal problems with zero flux boundary conditions. For an example of an application, see \cite{cappanera2024analysis}, where again a nonlocal Gray-Scott model is analyzed.

Lastly, we modified the original Klausmeier model (\ref{eqn:original_Klaus}) to give a more realistic description of plant-water dynamics in semiarid regions. In our work, we assumed that 1) water diffuses in the soil, 2) we supposed that plant growth is limited by soil resources, and 3) we accounted for long-range seed dispersal. Further modifications that are possible include better descriptions for
water flow. For instance, one could replace the linear diffusion term, \( w_{xx} \), with a nonlinear diffusion term, \( (w^{\rho})_{xx} \), \( \rho > 1 \)  (see \cite{vanderStelt2013, vazquez2006porous}), or even with a nonlinear, nonlocal diffusion term (see, e.g., \cite{biler2015nonlocal, vazquez2012nonlineardiffusion, Caffarelli_2011}) in order to model soil as a porous medium. This line of work leads to the theory of degenerate parabolic equations and would change the character of the estimates represented in the present work. However, this is an interesting research question that will broaden the application of our model to be closer to the actual dynamics of plants and water.

\bibliographystyle{unsrt}  
\bibliography{ref}

\begin{thebibliography}{10}

\bibitem{klausmeier1999}
Christopher~A. Klausmeier.
\newblock Regular and irregular patterns in semiarid vegetation.
\newblock {\em Science}, 284(5421):1826--1828, 1999.

\bibitem{vanderStelt2013}
Sjors van~der Stelt, Arjen Doelman, Geertje Hek, and Jens D.~M. Rademacher.
\newblock Rise and fall of periodic patterns for a generalized
  {K}lausmeier--{G}ray--{S}cott model.
\newblock {\em Journal of Nonlinear Science}, 23(1):39--95, 2013.

\bibitem{sewalt2017}
Lotte Sewalt and Arjen Doelman.
\newblock Spatially periodic multipulse patterns in a generalized
  {K}lausmeier--{G}ray--{S}cott model.
\newblock {\em SIAM Journal on Applied Dynamical Systems}, 16(2):1113--1163,
  2017.

\bibitem{sherratt2018}
Lukas Eigentler and Jonathan~A. Sherratt.
\newblock Analysis of a model for banded vegetation patterns in semi-arid
  environments with nonlocal dispersal.
\newblock {\em Journal of Mathematical Biology}, 77(3):739--763, 2018.

\bibitem{bastiaansen2019}
Robbin Bastiaansen, Paul Carter, and Arjen Doelman.
\newblock Stable planar vegetation stripe patterns on sloped terrain in dryland
  ecosystems.
\newblock {\em Nonlinearity}, 32(8):2759, jul 2019.

\bibitem{BROMLEY19971}
J.~Bromley, J.~Brouwer, A.P. Barker, S.R. Gaze, and C.~Valentine.
\newblock The role of surface water redistribution in an area of patterned
  vegetation in a semi--arid environment, south-west {N}iger.
\newblock {\em Journal of Hydrology}, 198(1):1--29, 1997.

\bibitem{AGUIAR1999273}
Martin~R. Aguiar and Osvaldo~E. Sala.
\newblock Patch structure, dynamics and implications for the functioning of
  arid ecosystems.
\newblock {\em Trends in Ecology \& Evolution}, 14(7):273--277, 1999.

\bibitem{LEPRUN199925}
Jean~Claude Leprun.
\newblock The influences of ecological factors on tiger bush and dotted bush
  patterns along a gradient from {M}ali to northern {B}urkina {F}aso.
\newblock {\em CATENA}, 37(1):25--44, 1999.

\bibitem{couteron2001}
P.~Couteron and O.~Lejeune.
\newblock Periodic spotted patterns in semi-arid vegetation explained by a
  propagation-inhibition model.
\newblock {\em Journal of Ecology}, 89(4):616--628, 2001.

\bibitem{macfadyen1950}
W.~A. Macfadyen.
\newblock Vegetation patterns in the semi-desert plains of {B}ritish
  {S}omaliland.
\newblock {\em The Geographical Journal}, 116(4/6):199--211, 1950.

\bibitem{meron2001}
J.~von Hardenberg, E.~Meron, M.~Shachak, and Y.~Zarmi.
\newblock Diversity of vegetation patterns and desertification.
\newblock {\em Phys. Rev. Lett.}, 87:198101, Oct 2001.

\bibitem{wilcox2003}
Bradford~P. Wilcox, David~D. Breshears, and Craig~D. Allen.
\newblock Ecohydrology of a resource-conserving semiarid woodland: Effects of
  scale and disturbance.
\newblock {\em Ecological Monographs}, 73(2):223--239, 2003.

\bibitem{Sherratt_2010}
Jonathan~A Sherratt.
\newblock Pattern solutions of the {K}lausmeier model for banded vegetation in
  semi-arid environments {I}.
\newblock {\em Nonlinearity}, 23(10):2657, aug 2010.

\bibitem{sherratt2011}
Jonathan~A. Sherratt.
\newblock Pattern solutions of the {K}lausmeier model for banded vegetation in
  semi-arid environments {II}: patterns with the largest possible propagation
  speeds.
\newblock {\em Proceedings of the Royal Society A: Mathematical, Physical and
  Engineering Sciences}, 467(2135):3272--3294, 2011.

\bibitem{sherratt2013}
Jonathan~A. Sherratt.
\newblock Pattern solutions of the {K}lausmeier model for banded vegetation in
  semi-arid environments {III}: The transition between homoclinic solutions.
\newblock {\em Physica D: Nonlinear Phenomena}, 242(1):30--41, 2013.

\bibitem{sherratt2013_2}
Jonathan~A. Sherratt.
\newblock Pattern solutions of the {K}lausmeier model for banded vegetation in
  semiarid environments {IV}: Slowly moving patterns and their stability.
\newblock {\em SIAM Journal on Applied Mathematics}, 73(1):330--350, 2013.

\bibitem{sherratt2013_3}
Jonathan~A. Sherratt.
\newblock Pattern solutions of the {K}lausmeier model for banded vegetation in
  semiarid environments {V}: The transition from patterns to desert.
\newblock {\em SIAM Journal on Applied Mathematics}, 73(4):1347--1367, 2013.

\bibitem{worral1959}
G.~A. Worrall.
\newblock The butana grass patterns.
\newblock {\em Journal of Soil Science}, 10(1):34--53, 1959.

\bibitem{montana1990}
C.~Montaña, J.~Lopez-Portillo, and A.~Mauchamp.
\newblock The response of two woody species to the conditions created by a
  shifting ecotone in an arid ecosystem.
\newblock {\em Journal of Ecology}, 78(3):789--798, 1990.

\bibitem{belsky1986population}
A.~Joy Belsky.
\newblock Population and community processes in a mosaic grassland in the
  {S}erengeti, {T}anzania.
\newblock {\em Journal of Ecology}, 74(3):841--856, 1986.

\bibitem{white1970brousse}
L.~P. White.
\newblock Brousse tigrée patterns in {S}outhern {N}iger.
\newblock {\em Journal of Ecology}, 58(2):549--553, 1970.

\bibitem{dunkerley2002oblique}
D.L. Dunkerley and K.J. Brown.
\newblock Oblique vegetation banding in the {A}ustralian arid zone:
  implications for theories of pattern evolution and maintenance.
\newblock {\em Journal of Arid Environments}, 51(2):163--181, 2002.

\bibitem{gilad2007mathematical}
E.~Gilad, J.~{von Hardenberg}, A.~Provenzale, M.~Shachak, and E.~Meron.
\newblock A mathematical model of plants as ecosystem engineers.
\newblock {\em Journal of Theoretical Biology}, 244(4):680--691, 2007.

\bibitem{gilad2004ecosystem}
E.~Gilad, J.~von Hardenberg, A.~Provenzale, M.~Shachak, and E.~Meron.
\newblock Ecosystem engineers: From pattern formation to habitat creation.
\newblock {\em Phys. Rev. Lett.}, 93:098105, Aug 2004.

\bibitem{bennett2019long}
Jamie~J.R. Bennett and Jonathan~A. Sherratt.
\newblock Long-distance seed dispersal affects the resilience of banded
  vegetation patterns in semi-deserts.
\newblock {\em Journal of Theoretical Biology}, 481:151--161, 2019.
\newblock Celebrating the 60th Birthday of Professor Philip Maini.

\bibitem{vonhardenberg2001diversity}
J.~von Hardenberg, E.~Meron, M.~Shachak, and Y.~Zarmi.
\newblock Diversity of vegetation patterns and desertification.
\newblock {\em Phys. Rev. Lett.}, 87:198101, Oct 2001.

\bibitem{hillel2014environmental}
Daniel Hillel.
\newblock {\em Environmental soil physics: fundamentals, applications, and
  environmental considerations}.
\newblock Elsevier Science, 2014.

\bibitem{bullock2017synthesis}
James~M. Bullock, Laura Mallada~González, Riin Tamme, Lars Götzenberger,
  Steven~M. White, Meelis Pärtel, and Danny A.~P. Hooftman.
\newblock A synthesis of empirical plant dispersal kernels.
\newblock {\em Journal of Ecology}, 105(1):6--19, 2017.

\bibitem{pueyo2008}
Y.~Pueyo, S.~Kefi, C.~L. Alados, and M.~Rietkerk.
\newblock Dispersal strategies and spatial organization of vegetation in arid
  ecosystems.
\newblock {\em Oikos}, 117(10):1522--1532, 2008.

\bibitem{du2012analysis}
Qiang Du, Max Gunzburger, R.~B. Lehoucq, and Kun Zhou.
\newblock Analysis and approximation of nonlocal diffusion problems with volume
  constraints.
\newblock {\em SIAM Review}, 54(4):667--696, 2012.

\bibitem{cappanera2024analysis}
Loic Cappanera, Gabriela Jaramillo, and Cory Ward.
\newblock Analysis and simulation of a nonlocal {G}ray–{S}cott model.
\newblock {\em SIAM Journal on Applied Mathematics}, 84(3):856--889, 2024.

\bibitem{burkovska2021}
Olena Burkovska and Max Gunzburger.
\newblock On a nonlocal {C}ahn--{H}illiard model permitting sharp interfaces.
\newblock {\em Mathematical Models and Methods in Applied Sciences},
  31(09):1749--1786, 2021.

\bibitem{Philippe2023}
Philippe Lauren\c{c}ot and Christoph Walker.
\newblock A nonlocal {G}ray-{S}cott model: {W}ell--posedness and diffusive
  limit.
\newblock {\em Discrete and Continuous Dynamical Systems - S},
  16(12):3709--3732, 2023.

\bibitem{evans2022partial}
Lawrence~C Evans.
\newblock {\em Partial differential equations}, volume~19.
\newblock American Mathematical Society, 2022.

\bibitem{aubin1963}
Jean-Pierre Aubin.
\newblock Analyse mathematique-un theoreme de compacite.
\newblock {\em Comptes Rendus Hebdomadaires Des Seances De L Academie Des
  Sciences}, 256(24):5042--5044, 1963.

\bibitem{adams2003sobolev}
R.A. Adams and J.J.F. Fournier.
\newblock {\em Sobolev Spaces}.
\newblock Pure and Applied Mathematics. Academic Press, 2003.

\bibitem{folland1999real}
Gerald~B Folland.
\newblock {\em Real analysis: modern techniques and their applications},
  volume~40.
\newblock John Wiley \& Sons, 1999.

\bibitem{vazquez2006porous}
Juan~Luis Vazquez.
\newblock {\em The Porous Medium Equation: Mathematical Theory}.
\newblock Oxford University Press, 10 2006.

\bibitem{biler2015nonlocal}
Piotr Biler, Cyril Imbert, and Grzegorz Karch.
\newblock The nonlocal porous medium equation: {B}arenblatt profiles and other
  weak solutions.
\newblock {\em Archive for Rational Mechanics and Analysis}, 215:497--529,
  2015.

\bibitem{vazquez2012nonlineardiffusion}
Juan~Luis V{\'a}zquez.
\newblock Nonlinear diffusion with fractional laplacian operators.
\newblock In Helge Holden and Kenneth~H. Karlsen, editors, {\em Nonlinear
  Partial Differential Equations}, pages 271--298, Berlin, Heidelberg, 2012.
  Springer Berlin Heidelberg.

\bibitem{Caffarelli_2011}
Luis Caffarelli and Juan~Luis Vazquez.
\newblock Nonlinear porous medium flow with fractional potential pressure.
\newblock {\em Archive for Rational Mechanics and Analysis}, 202(2):537–565,
  April 2011.

\end{thebibliography}

\end{document}